\documentclass[12pt,reqno]{amsart}

\usepackage{a4wide}
\usepackage{amsmath}
\usepackage{psfrag}
\usepackage{amsthm}
\usepackage{amssymb,amsxtra}
\usepackage{appendix}
\usepackage[usenames]{color}
\usepackage{stmaryrd}
\usepackage{mathrsfs}
\usepackage{upgreek}
\usepackage{ytableau}
\usepackage{float}
\usepackage{comment}
\usepackage{tikz-cd}

\usepackage{paralist}
\usepackage{graphicx}
\usepackage[all]{xy}
\usepackage{tikz}
\usetikzlibrary{matrix,fit}
\usepackage{geometry}
\usetikzlibrary{positioning}
\usepackage{hyperref}
\usepackage{cleveref}
\usetikzlibrary{decorations.pathreplacing}

 \allowdisplaybreaks

\newcommand{\Z}{\mathbb{Z}}
\newcommand{\NN}{\mathbb{N}}

\newcommand{\Rect}{\mathsf{R}}
\newcommand{\sym}[1]{\mathfrak{S}_{#1}}

\newcommand{\F}{{F}}

\renewcommand{\i}{\mathring{\imath}}
\renewcommand{\j}{\mathring{\jmath}}

\renewcommand{\geq}{\geqslant}
\renewcommand{\leq}{\leqslant}
\renewcommand{\k}{\ell}

\newcommand{\block}{b}

\newcommand{\M}{\mathsf{M}}

\theoremstyle{definition}

\newtheorem{thm}{Theorem}[section]
\newtheorem{lem}[thm]{Lemma}
\newtheorem{defn}[thm]{Definition}
\newtheorem{prop}[thm]{Proposition}
\newtheorem{cor}[thm]{Corollary}

\newtheorem{eg}[thm]{Example}

\numberwithin{equation}{section}

\allowdisplaybreaks

\begin{document}
\title{Tensor Powers of Indecomposable Modules for $\F \sym{p}$}

\author{Manzu Kua}
\address[M. Kua]{Division of Mathematical Sciences, Nanyang Technological University, SPMS-04-01, 21 Nanyang Link, Singapore 637371.}
\email{s220025@e.ntu.edu.sg}

\author{Kay Jin Lim}
\address[K. J. Lim]{Division of Mathematical Sciences, Nanyang Technological University, SPMS-04-01, 21 Nanyang Link, Singapore 637371.}
\email{limkj@ntu.edu.sg}

\begin{abstract} In the earlier paper \cite{kualim1}, the authors gave an explicit tensor product formula, modulo projectives, for modules over the group algebra $\F\sym{p}$. In this paper, we use this formula to study tensor powers of such modules. In certain cases, we give combinatorial descriptions of the multiplicities of the indecomposable summands appearing in their direct-sum decompositions. We also study the asymptotic behaviour of these multiplicities as the underlying prime tends to infinity.

\end{abstract}


\thanks{The second author is supported by Singapore Ministry of Education AcRF Tier 2 grant MOE-T2EP20225-0003.}

\maketitle

\section{Introduction} 

Computing tensor products of modules over a Hopf algebra is a notoriously difficult problem. Even for group algebras, there is in general no uniform method for decomposing a tensor product into indecomposable summands. In the semisimple case, this problem is closely related to the multiplication of irreducible characters. In the modular case, however, the situation becomes substantially more complicated, since group algebras often have infinitely many indecomposable modules and tensor products need not behave semisimply.

Let $F$ be an algebraically closed field of characteristic $p$. A classical example is given by cyclic groups in characteristic $p$. If $C_{p^s}$ is cyclic of order $p^s$, then the indecomposable modules over a field of characteristic $p$ are precisely the modules $F[x]/(x^i)$, for $1 \leq i \leq p^s$, equivalently the Jordan blocks of size at most $p^s$ \cite{Higman:1954}. The tensor product decomposition of two such modules was described by Srinivasan \cite{Srinivasan:1964}. More generally, Dade \cite{Dade:1966} classified the indecomposable modules for groups with a normal cyclic Sylow $p$-subgroup, and tensor products of these modules were studied by Feit \cite{Feit:1966} and Lindsey \cite{Lindsey:1974}. Janusz \cite{Janusz:1966} also constructed indecomposable modules for groups with cyclic Sylow $p$-subgroups. These results show that, for groups with cyclic Sylow $p$-subgroups, tensor product questions have a rich structure, although explicit decomposition formulas can still be subtle.

For the symmetric group, tensor product problems are already highly non-trivial in the ordinary case. The irreducible modules are the Specht modules, and the multiplicities occurring in tensor products of Specht modules are the Kronecker coefficients. These coefficients play an important role in the theory of symmetric functions, where tensor products correspond to the internal product, but they are notoriously difficult to compute. In the modular case, further complications arise. The blocks of the symmetric groups were classified by Brauer and Robinson \cite{Brauer:1947,robinsonproof}, confirming the Nakayama conjecture: two Specht modules lie in the same block if and only if their corresponding partitions have the same $p$-core. Moreover, Specht modules are no longer simple in general. Instead, the simple modules are parametrised by the $p$-regular partitions and occur as the heads of certain Specht modules constructed by James \cite{james}. Tensor products of modular Specht modules and simple modules remain poorly understood outside a small number of special cases, such as tensoring with the signature representation, where the Mullineux map describes the effect on simple modules \cite{Ford/Kleshchev:1997,Mullineux:1979}.

In \cite{kualim1}, the authors studied tensor products of indecomposable modules for the group algebra $\F\sym{p}$ modulo projectives. It is known that the non-projective indecomposable $\F\sym{p}$-modules are given by the Heller translates of the simple modules in the principal block of $\F\sym{p}$. More precisely, up to isomorphism, they are $\Omega^r(D_k)$, where $r,k\in [0,p-2]$ and $D_k$ is the simple module labelled by the hook partition $(p-k,1^k)$. Since $\sym{p}$ has a cyclic Sylow $p$-subgroup of order $p$, the general theory of modules with cyclic Sylow $p$-subgroups applies in principle. Nevertheless, obtaining explicit decompositions requires additional information. %In particular, in that paper, the authors showed that the tensor product of two simple modules is semisimple modulo projectives.

The purpose of the present paper is to exploit the tensor product formula for the study of tensor powers of these indecomposable $\F\sym{p}$-modules. While the tensor product of two indecomposable modules gives the basic multiplication rule, tensor powers lead naturally to an iterated decomposition problem. Namely, for an indecomposable module $V$, one may ask for the multiplicity of each indecomposable non-projective direct summand in
$V^{\otimes n}$. The formulas obtained in \cite{kualim1} reduce this question to a finite combinatorial problem governed by the tensor product rules in the stable Green ring of $\F\sym{p}$. 

More precisely, in Section \ref{S:Mk}, we introduce an auxiliary matrix $\M_k$ whose powers encode the information about the tensor powers of the simple module $D_k$ (or more generally, $\Omega^r(D_k)$ for any $r\in [0,2p-3]$). In Sections \ref{S:Path} and \ref{S:Wave}, we give the combinatorial descriptions for the multiplicities of indecomposable summands in the decomposition of the tensor powers of $\Omega^r(D_1)$, for $k=1$, and $\Omega^r(D_2)$, for $k=2$ respectively. These multiplicities are described in terms of the walks in the path graphs and wave sequences. In Section \ref{S:Stability}, we study the asymptotic behaviour of the multiplicities as $p$ tends to infinity.

\section{Preliminaries} 

We refer the reader to \cite{james,jk} for the representation theory of symmetric groups. For basic knowledge of the representation theory of finite-dimensional algebras, we refer the reader to \cite{alp,ben1}. Throughout, let $\F$ be an algebraically closed field of characteristic $p\geq 3$ and $\sym{n}$ be the symmetric group acting on $n$ elements. For integers $a\leq b$, we write $[a,b]$ for the integer interval between $a$ and $b$, that is, it consists of all integers $c$ such that $a\leq c\leq b$. 

Let $G$ be a finite group and $\F G$ be the group algebra. For an $\F G$-module $V$ and $r\in\Z$, $\Omega^r(V)$ is the $r$th Heller translate of $V$. For another $\F G$-module $W$, we write $V\oplus W$ for the direct sum and $V\otimes W$ for their tensor product (over $\F$) with the action defined by the coproduct of $\F G$ given by $\Delta(g)=g\otimes g$. Objects in the stable Green ring for the group algebra are denoted as $[V]$. Furthermore, we write $V\equiv W$ if $V\oplus P\cong W\oplus Q$ for some projective $\F G$-modules $P,Q$, that is, $[V]=[W]$. Finally, if $W$ is also indecomposable, we write $[V:W]$ for the multiplicity of $W$ as a direct summand of $V$. 

We now consider the symmetric group algebra $\F\sym{n}$. The simple modules for $\F\sym{n}$ are parametrised by $p$-regular partitions of $n$. In the literature, they are labelled by $D^\lambda$ and can be identified as the quotient of a certain Specht module by its unique maximal submodule. On the other hand, the block decomposition of the group algebra $\F\sym{n}$ is parametrised by the $p$-cores of partitions of $n$ (proved by Brauer and Robinson \cite{Brauer:1947,robinsonproof} but it is also known as the Nakayama Conjecture). More precisely, the simple modules $D^\lambda,D^\mu$ belong in the same block if and only if $\lambda,\mu$ have the same $p$-core. For example, in the case when $n<p$, a partition of $n$ is the same as its $p$-core and hence all simple modules are projective. 

In this paper, the objects of study are modules for $\F\sym{p}$. A partition $\lambda$ of $p$ is different from its $p$-core if and only if $\lambda$ is a hook of size $p$. In this case, the $p$-core of $\lambda$ is the empty partition $\varnothing$, $\lambda=(p-i,1^i)$ for some $i\in [0,p-2]$ and $D^\lambda$ belongs to the principal block $\block_0$ of $\F\sym{p}$. For this reason, we write $D_i$ for $D^\lambda$ if $\lambda=(p-i,1^i)$ and let $P_i$ be the projective cover of $D_i$. Moreover, we have the following theorem. 

\begin{thm}
  A complete set of representatives of the isomorphism classes of indecomposable $\F\sym{p}$-modules is given by \[\{\Omega^r(D_j):r,j\in[0,p-2]\}\cup \{P_j:j\in [0,p-2]\}\cup \{D^\mu:\mu\not\in \block_0\}.\]
\end{thm}

The module $D^\mu$ with $\mu\not\in \block_0$ is simple projective. So the stable Green ring of the group algebra $\F\sym{p}$ has a $\Z$-basis $\{[\Omega^r(D_j)]:r,j\in[0,p-2]\}$. In fact, $D_j$ has period $2p-2$ and $\Omega^{p-1}(D_j)\cong D_{p-2-j}$. In the earlier paper \cite{kualim1}, the authors gave an explicit decomposition formula for tensor products of indecomposable $\F\sym{p}$-modules modulo projectives. To describe the formula, we need the following notation.

Consider the grid of squares with the horizontal (respectively, vertical) lines labelled by $0,1,\ldots,p$ from top to bottom (respectively, $0,1,\ldots,p-2$ from left to right). Each grid point is labelled by $(a,b)$ if it belongs to the horizontal and vertical lines labelled by $a$ and $b$ respectively. For $j \in [0, p-2]$, consider the rectangle with vertices the grid points $(0,j)$, $(j,0)$, $(p-2,p-2-j)$ and $(p-2-j,p-2)$. We call it the $j$-rectangle. Furthermore, let 
\begin{align*}
	\ell_{i,j} &= \left \{\begin{array}{ll}
		j-i &\text{if $i \in [0,j]$,} \\ 
		i-j&\text{if $i \in [j+1, p-1]$,} 
	\end{array}\right . \\
	r_{i,j} &= \left \{\begin{array}{ll}
		j+i &\text{if $i \in [0, p-2-j]$,}\\
		2p-j-4-i&\text{if $i \in [p-j-1, p-1]$.}
	\end{array}\right .
\end{align*} The numbers $\ell_{i,j},r_{i,j}$ are the left and right boundaries of the $j$-rectangle at level $i$. Finally, we let \[\Rect(i,j)=\{\ell_{i,j}+2k:k\in [0,(r_{i,j}-\ell_{i,j})/2]\},\] that is, $\Rect(i,j)$ consists of all the grid points of the $j$-rectangle at level $i$. 

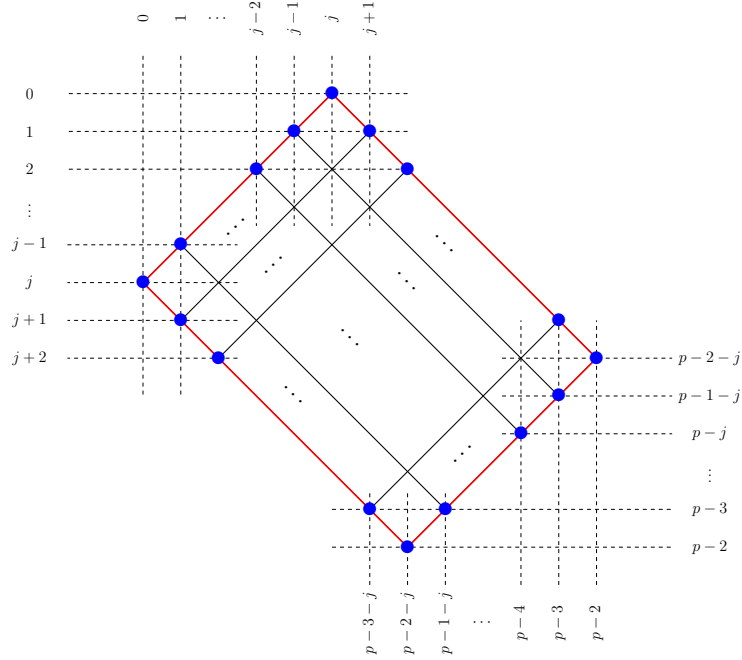
\begin{figure}[H]
\[\scalebox{0.5}{
\begin{tikzpicture}
\foreach \i in {2,3}
{\draw[dashed] (\i,8) -- (\i,-1);
\draw[dashed] (\i+3,8) -- (\i+3,3.5);
\draw[dashed] (\i+5,8) -- (\i+5,3.5);
\draw[dashed] (\i+11,-6) -- (\i+11,1);
\draw[dashed] (\i+6,-6) -- (\i+6,-3.5);}

\draw[dashed] (10,-6) -- (10,-3.5);
\draw[dashed] (12,-6) -- (12,1);

\foreach \i in {0,1,2}
{\draw[dashed] (9,\i+5) -- (0,\i+5);
\draw[dashed] (4.5,\i) -- (0,\i);
\draw[dashed] (11.5,\i-2) -- (16,\i-2);}

\draw[dashed] (7,-5)--(16,-5);
%\draw[dashed] (7,\i-7) -- (16,\i-7);}

\draw[dashed] (0,3) -- (4.5,3);
\draw[dashed] (7,-4) -- (16,-4);

\foreach \i in {2,3,5,6,7}
{\draw[thick] (\i,\i)--(\i+7,\i-7);}

\foreach \j in {0,1,2,6,7}
{\draw[thick] (\j+2,-\j+2)--(\j+7,-\j+7);}

\draw[very thick,color=red] (7,7)--(2,2)--(9,-5)--(14,0)--cycle;

\foreach \k in {1,2.5,4,5}
{\node at (\k+5,\k-2) {\rotatebox{-45}{\Large$\cdots$}};}

\foreach \l in {0,1,6}
{\node at (\l+4.5,-\l+3.5) {\rotatebox{45}{\Large$\cdots$}};}

%\draw[fill=gray!80] (1,1)--(2,0)--(3,1)--(2,2)--cycle;
%\draw[fill=gray!80] (13,-1)--(14,-2)--(15,-1)--(14,0)--cycle;

\foreach \i/\label in {1/0,2/1,4/j-2,5/j-1,6/j,7/j+1}
{\node at (\i+1,9) {\rotatebox{90}{$\label$}};}

\foreach \i/\label in {8/p-3-j,9/p-2-j,10/p-1-j,12/p-4,13/p-3,14/p-2}
{\node at (\i,-7) {\rotatebox{90}{$\label$}};}

\foreach \k/\label in {7/0, 6/1, 5/2, 3/j-1, 2/j, 1/j+1, 0/j+2}
{\node at (-1,\k) {$\label$};}

\foreach \k/\label in {-5/p-2, -4/p-3, -2/p-j, -1/p-1-j, 0/p-2-j}
{\node at (17,\k) {$\label$};}

\foreach \i/\j in {7/7, 6/6, 8/6, 5/5, 9/5, 3/3, 2/2, 3/1, 4/0, 8/-4, 9/-5, 10/-4, 12/-2, 13/-1, 14/0, 13/1}
{\node at (\i,\j) {\textcolor{blue}{\huge{$\bullet$}}};}

\node at (-1,4) {$\vdots$};
\node at (17,-3) {$\vdots$};
\node at (4,9) {$\cdots$};
\node at (11,-7) {$\cdots$};
\end{tikzpicture}}
\]
\caption{$j$-diagram}
\label{Fig:diagram}
\end{figure}

We end this section with the following tensor product formula in the stable Green ring of $\F\sym{p}$. 

\begin{thm}[{\cite[Theorem 4.2]{kualim1}}]\label{T:StableGreenRing} Let $i,j,k,\ell\in [0,p-2]$. Then $\Omega^k(D_i)\otimes \Omega^\ell(D_j)\equiv \bigoplus_{t\in\Rect(i,j)}\Omega^{k+\ell}(D_t)$.
\end{thm}

\section{The matrix $\M_k$}\label{S:Mk}

In this section, we define the matrix $\M_k$ we shall be studying throughout this paper. Its powers $\M_k^b$ encode the information about the decomposition of the even tensor powers $D_k^{\otimes 2b}$ of the simple module $D_k$ (see Theorem \ref{T:PowersofMk} below). For odd powers of $D_k$, we can make use of Proposition \ref{P:reduction}. 

To begin, we make the following simple observations using the $j$-rectangle. 

\begin{lem}\label{L:Rect}
    Let $i,j\in [0,p-2]$. 
    \begin{enumerate}[(i)]
        \item If both $i,j$ have the same parity, then $\Rect(i,j)$ consists of contiguous even integers.
        \item If $i,j$ have opposite parities, then $\Rect(i,j)$ consists of contiguous odd integers.
    \end{enumerate} In particular, for $2r,2s\in [0,p-2]$, we have $[D_{2r}\otimes D_{2s}:D_\ell]=0$ for any odd integer $\ell\in [0,p-2]$.
\end{lem}
\begin{proof}
    Parts (i) and (ii) are obvious from the definition of $j$-rectangle. The final assertion follows from Theorem \ref{T:StableGreenRing}.
\end{proof}

The following proposition allow us to focus only on the even powers of $D_k$.  
%In the decomposition of the $n$th tensor powers $D_k^{\otimes n}$ of $D_k$, the occurrences of the simple modules as direct summands depend on the parities of $n$ and $k$. To keep the presentation simple, we have the following proposition which allow us to focus only on the case when $n$ is even. 

\begin{prop}\label{P:reduction} Let $k\in [0,p-2]$ and $n\in\NN_0$. Suppose that $D^{\otimes n}_k\equiv \bigoplus_{i=0}^{p-2}a_iD_i$. Then \[D^{\otimes (n+1)}_k
\equiv \bigoplus_{j=0}^{p-2}b_j D_j\] where $b_j=\sum_{j\in\Rect(k,i)} a_i$, that is, the sum of the $a_i$'s such that $j$ belongs to $\Rect(k,i)$. In particular, 
\begin{enumerate}[(i)]
    \item if $k$ or $n$ is even, then $a_i=0$ for every odd integer $i$;
    \item if both $k$ and $n$ are odd, then $a_i=0$ for every even integer $i$. 
\end{enumerate}
\end{prop}
\begin{proof}
    By Theorem \ref{T:StableGreenRing}, we have 
    \begin{align*}
        D^{\otimes (n+1)}_k&\equiv D_k\otimes \bigoplus_{i=0}^{p-2}a_iD_i\equiv \bigoplus_{i=0}^{p-2}a_i\bigoplus_{j\in\Rect(k,i)}D_j=\bigoplus_{j=0}^{p-2}\left (\sum_{j\in\Rect(k,i)} a_i\right ) D_j.
    \end{align*} We now prove parts (i) and (ii) by induction. Assume first that $k$ is even. Let $j$ be an odd integer. For $j\in \Rect(k,i)$, it is necessary that $i$ is odd by Lemma \ref{L:Rect}. So $b_j=\sum_{j\in\Rect(k,i)} a_i=0$. Suppose now that $k$ is odd. The conclusion is clear when $n=0,1$. If $n$ is even, by induction hypothesis, $a_i=0$ for every odd integer $i$. Suppose that $j$ is even. Similarly as before, for $j\in \Rect(k,i)$, it is necessary that $i$ is odd. So $b_j=0$. The case for $n$ is odd is similar. 
\end{proof}

In view of Proposition \ref{P:reduction}, in order to understand the decomposition of the tensor powers of $D_k$, we can assume that $n$ is even. For $n$ is even, suppose that 
\begin{equation}\label{Eq:Dkn}
D_k^{\otimes n}\equiv \bigoplus_{0\leq 2j\leq p-3}s_{n,2j}^{(k)}D_{2j}.
\end{equation} On the other hand, using Theorem \ref{T:StableGreenRing}, we have \[D_k^{\otimes 2}\equiv \bigoplus_{t\in\Rect(k,k)} D_t\equiv \bigoplus_{i=0}^m D_{2i}\] where $m=\min\{k,p-2-k\}$. Therefore,  
\begin{equation}\label{Eq:Dkn+2} D_k^{\otimes (n+2)}\cong D_k^{\otimes 2}\otimes D_k^{\otimes n}\equiv \bigoplus_{i=0}^mD_{2i}\otimes \bigoplus_{0\leq 2j\leq p-3}s_{n,2j}^{(k)}D_{2j}\equiv \bigoplus s_{n,2j}^{(k)} [D_{2i}\otimes D_{2j}:D_{2\ell}]D_{2\ell}
\end{equation} where the last equality follows from Lemma \ref{L:Rect} and the sum for the last item runs over all $2j,2\ell\in [0,p-3]$ and $i\in [0,m]$. In view of Equations \ref{Eq:Dkn} and \ref{Eq:Dkn+2}, we define the following notation. 

\begin{defn}\label{D:Mvnk}
    Let $k\in [0,p-2]$, $m=\min\{k,p-2-k\}$ and $n\in\NN_0$ be even. Let $v_n^{(k)}$ be the column matrix (see Equation \ref{Eq:Dkn}) 
    \[v_n^{(k)}=\begin{bmatrix}
            s_{n,0}^{(k)}\\ s_{n,2}^{(k)}\\ \vdots\\ s_{n,p-3}^{(k)} 
    \end{bmatrix}.\] We can simplify the notation of $s_{n,2j}^{(k)}$ and $v_n^{(k)}$ by suppressing the superscript. Furthermore, let $\M_k$ be the matrix with both row and column labelled by $[0,\frac{p-3}{2}]$ such that, for each $\ell,j\in [0,\frac{p-3}{2}]$, the $(\ell,j)$-entry of $\M_k$ is \[(\M_k)_{\ell,j}=\sum_{i=0}^m[D_{2i}\otimes D_{2j}:D_{2\ell}].\] For example, the top left entry of $\M_k$ is called the $(0,0)$-entry. More explicitly, the entry $(\M_k)_{\ell,j}$ is the total number of points in the column of the $2j$-rectangle labelled by $2\ell$ and above (and including) the $2m$th layer. 
\end{defn}

The following is an example illustrating how we read off the entries of the matrix $\M_k$.

\begin{eg} Let $p=7$ and $k=1$. We have the following $2j$-rectangles one for each $j\in[0,2]$. For each $j\in [0,2]$, we have marked the points in its column labelled by 0 (respectively, 2 and 4) by blue (respectively, red and green) above (and including) its second layer.  
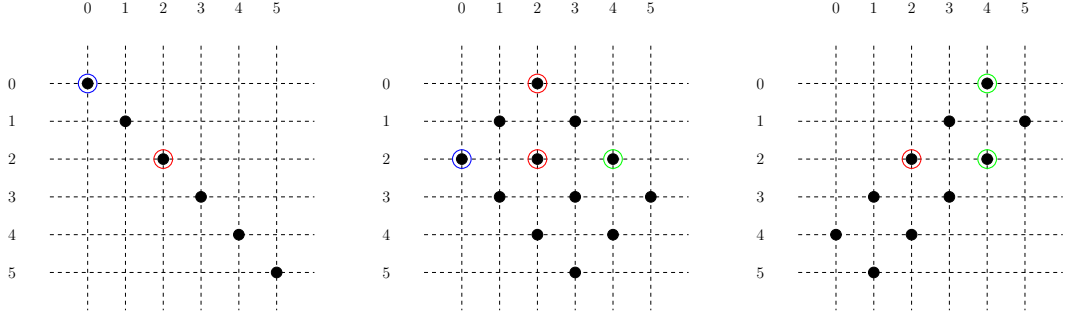
\[\scalebox{0.5}{\begin{tikzpicture}

\foreach \i in {0,1,2,3,4,5} {
	\draw[dashed] (\i,1)--(\i,-6);
	\node at (\i,2) {$\i$};
}

\foreach \i in {0,1,2,3,4,5} {
	\draw[dashed] (-1,-\i)--(6,-\i);
	\node at (-2,-\i) {$\i$};
}

\foreach \i in {0,1,2,3,4,5} {
	\filldraw (\i,-\i) circle (4pt);
}

\draw[red,thick] (2,-2) circle (7pt);
\draw[blue,thick] (0,0) circle (7pt);
\end{tikzpicture}}  \qquad \scalebox{0.5}{\begin{tikzpicture}

\foreach \i in {0,1,2,3,4,5} {
	\draw[dashed] (\i,1)--(\i,-6);
	\node at (\i,2) {$\i$};
}

\foreach \i in {0,1,2,3,4,5} {
	\draw[dashed] (-1,-\i)--(6,-\i);
	\node at (-2,-\i) {$\i$};
}

\foreach \i in {0,1,2,3} {
	\filldraw (\i+2,-\i) circle (4pt);
	\filldraw (\i+1, -\i-1) circle (4pt);
	\filldraw (\i, -\i-2) circle (4pt);
}

\draw[red,thick] (2,0) circle (7pt);
\draw[red,thick] (2,-2) circle (7pt);
\draw[blue,thick] (0,-2) circle (7pt);
\draw[green,thick] (4,-2) circle (7pt);
\end{tikzpicture}}  \qquad \scalebox{0.5}{\begin{tikzpicture}

\foreach \i in {0,1,2,3,4,5} {
  \draw[dashed] (\i,1)--(\i,-6);
  \node at (\i,2) {$\i$};
}

\foreach \i in {0,1,2,3,4,5} {
  \draw[dashed] (-1,-\i)--(6,-\i);
  \node at (-2,-\i) {$\i$};
}

\foreach \i in {0,1,2,3,4} {
  \filldraw (\i,\i-4) circle (4pt);
  \filldraw (\i+1, \i-5) circle (4pt);
}

\draw[red,thick] (2,-2) circle (7pt);
\draw[green,thick] (4,-2) circle (7pt);
\draw[green,thick] (4,0) circle (7pt);
\end{tikzpicture}}\]

Therefore, the matrix $\M_1$ is given as follows:
\[\M_1 = \begin{bmatrix}
	1 & 1 & 0 \\
	1 & 2 & 1 \\
	0 & 1 & 2
\end{bmatrix}.\]
\end{eg}

\begin{lem}\label{L:MkSymm}
  The matrix $\M_k$ is symmetric.
\end{lem}
\begin{proof}
   Let $\ell,j\in [0,\frac{p-3}{2}]$ and $m=\min\{k,p-2-k\}$. By definition, we aim to prove that, for any $i\in [0,m]$, \[[D_{2i}\otimes D_{2j}:D_{2\ell}]=[D_{2i}\otimes D_{2\ell}:D_{2j}].\] Since $D_{2i}\otimes D_{2j}\cong D_{2j}\otimes D_{2i}$ (and similar for the other one), we aim to check that $2\ell\in \Rect(2j,2i)$ if and only if $2j\in\Rect(2\ell,2i)$. For this, we consider grid points with $2i$-rectangle (in blue in the figure below). The point $(2j,2\ell)$ is the image of the reflection of $(2\ell,2j)$ about the diagonal line passing through $(0,0)$ (see the red line). The line is also the line of symmetry of the (grid points of the) $2i$-rectangle. Therefore, $(2j,2\ell)$ belongs to the $2i$-rectangle if and only if $(2\ell,2j)$ belongs. So $2\ell\in \Rect(2j,2i)$ if and only if $2j\in\Rect(2\ell,2i)$.
\[\small{
\begin{tikzpicture}[xscale=.35,yscale=.35]
\fill[blue,rounded corners=1mm,opacity=0.2] (0,0)--(4,4)--(10,-2)--(6,-6)--(0,0)--(4,4)--cycle;

\draw[thick,red] (0,4)--(10,-6);
\node at (0,4) {$\bullet$};
\node at (4,4) {$\bullet$};
\node at (0,0) {$\bullet$};
\draw[dashed] (0,4)--(10,4);
\draw[dashed] (0,4)--(0,-6);

\node at (4,5) {$(0,2i)$};
\node at (-2,0) {$(2i,0)$};
\node at (-2,4) {$(0,0)$};
\node at (5,-3) {$\bullet$};
\node at (5,-3.8) {$(2j,2\ell)$};
\node at (7,-1) {$\bullet$};
\node at (7,-.2) {$(2\ell,2j)$};
\end{tikzpicture}}\]
\end{proof}

The next theorem shows that powers of the matrix $\M_k$ encodes information about the decomposition of the tensor powers of $\Omega^r(D_k)$. For this, we need the following easy lemma. 

\begin{lem}\label{L:vnk}
    Let $k\in [0,p-2]$ and $n\in\NN_0$ be even. Then $v_{n+2}^{(k)}=\M_kv_n^{(k)}$.
\end{lem}
\begin{proof}
    This follows immediately using Equation \ref{Eq:Dkn+2} and our notation in Definition \ref{D:Mvnk}.
\end{proof}

\begin{thm}\label{T:PowersofMk}
     Let $k\in [0,p-2]$, $b\in\NN_0$ and $r\in [0,2p-3]$. We have \[(\Omega^r(D_k))^{\otimes 2b} \equiv \bigoplus_{0\leq 2\ell\leq p-3} (\M_k^b)_{\ell,0}\Omega^{2br}(D_{2\ell}).\] 
\end{thm}
\begin{proof}
    Using Lemma \ref{L:vnk} and induction, we have $v_{2b}^{(k)}=\M_k^bv_0^{(k)}$ where $v_0^{(k)}=\begin{bmatrix}
        1&0&\cdots&0
    \end{bmatrix}^\top$. So $v_{2b}^{(k)}$ is the first column of $\M_k^b$. The definition of $v_{2b}^{(k)}$ and Equation \ref{Eq:Dkn} now give us the desired statement. We obtain the last assertion because, modulo projectives, the Heller translate $\Omega^r$ commutes with direct sum and $\Omega^{2r}(V\otimes W)\equiv (\Omega^r V)\otimes (\Omega^r W)$.  
\end{proof}

We now give an example illustrating our theorem.

\begin{eg}\label{Eg:p7D12}
    Let $p=7$. Using Theorem \ref{T:StableGreenRing}, we have
    \begin{align*} 
    D_1^{\otimes 2} &\equiv D_0 \oplus D_2,& D_2^{\otimes 2}&\equiv D_0\oplus D_2\oplus D_4, \\
    D_1^{\otimes 4} &\equiv 2D_0 \oplus 3 D_2 \oplus D_4,& D_2^{\otimes 4}&\equiv 3D_0\oplus 6D_2\oplus 5D_4, \\
    D_1^{\otimes 6} &\equiv 5 D_0 \oplus 9 D_2 \oplus 5 D_4,& D_2^{\otimes 6}&\equiv 14D_0\oplus 31D_2\oplus 25D_4.
    \end{align*} On the other hand, 
    \begin{align*}
	\M_{1} &= \begin{bmatrix} 1 & 1 & 0 \\ 1 & 2 & 1 \\ 0 & 1 & 2 \end{bmatrix},& \M_1^2 &= \begin{bmatrix} 2 & 3  & 1\\3 & 6 & 4 \\ 1 & 4 & 5 \end{bmatrix}, &\M_1^3 &= \begin{bmatrix} 5 & 9 & 5 \\ 9 & 19 &  14 \\ 5 & 14 & 14 \end{bmatrix},\\ 
	\M_{2} &= \begin{bmatrix} 1 & 1 & 1 \\ 1 & 3 & 2 \\ 1 & 2 & 2 \end{bmatrix},& \M_2^2 &= \begin{bmatrix} 3 & 6  & 5 \\ 6 & 14 & 11 \\ 5 & 11 & 9 \end{bmatrix},&\M_2^3 &= \begin{bmatrix} 14 & 31 & 25 \\ 31 & 70 &  56 \\ 25 & 56 & 45 \end{bmatrix}.
	\end{align*}
\end{eg}

As an application of the theorem, we have the following corollary. %In this statement, for an $\F\sym{p}$-module $V$, we write $[V]$ for the object in the stable Green ring $\F\sym{p}$.

\begin{cor}\label{C:charpoly}
Let $k\in [0,p-2]$ and $f(X)$ be the characteristic polynomial of $\M_k$. Then $[D_k]$ satisfies $f(X^2)$. In particular, the unique minimal monic polynomial $g(X)$ such that $[D_k]$ satisfies divides $f(X^2)$. 
\end{cor}
\begin{proof} Suppose that $f(X)=X^n+c_{n-1}X^{n-1}+\cdots+c_1X+c_0$ where $c_0,c_1,\ldots,c_{n-1}\in\Z$. Let $c_n=1$. By Theorem \ref{T:PowersofMk}, we only need to check for the coefficient of $[D_{2\ell}]$ in $f([D_k^{\otimes 2}])$. For any $2\ell\in [0,p-2]$, by Theorem \ref{T:PowersofMk}, we have 
\begin{align*}
  ([D_k^{\otimes 2n}]+\cdots+c_1[D_k^{\otimes 2}]+c_0[D_k^0]: [D_{2\ell}])&=\sum_{i=0}^nc_i([D_k^{\otimes 2i}]:[D_{2\ell}])=\sum_{i=0}^nc_i(\M_k^{i})_{0,\ell}\\
  &=f(\M_k)_{0,\ell}=0.
\end{align*} Therefore, $f([D_k^{\otimes 2}])$ is virtually the zero module. %Since $\Omega^r$ commutes with tensor product and direct sum up to projectives, we also have $f([\Omega^r(D_k)^{\otimes 2}])$ is virtually zero. The final assertion is obvious. 
\end{proof}

We end this section with an example. 

\begin{eg}
  Let $p=7$. As in Example \ref{Eg:p7D12}, we also have 
  \begin{align*}
    D_1^{\otimes 3}\equiv 2D_1 \oplus D_3, \quad D_1^{\otimes 5}\equiv 5D_1\oplus 4D_3\oplus D_5.
  \end{align*} The unique minimal monic polynomial satisfied by $[D_1]$ is precisely \[f(X^2)=(X^2)^3-5(X^2)^2+6(X^2)-1\] where $f(X)=X^3-5X^2+6X-1$ is the characteristic polynomial of $\M_1$. So Corollary \ref{C:charpoly} is optimal. 
\end{eg}

\section{Walks on a Path Graph}\label{S:Path}

In the previous section, we showed that the entries of the first row (or column) of the matrix $\M_k^b$ give the multiplicities in the direct sum decomposition of $(\Omega^r(D_k))^{\otimes 2b}$. In this section, we examine particularly $\M_1^b$. More precisely, we give combinatorial descriptions for the entries of $\M_1^b$ and infer that for $(\Omega^r(D_1))^{\otimes 2b}$. The combinatorial object is the following path graph. 

\begin{defn}
    Let $r$ be a positive integer. A path graph $\mathsf{P}_r$ on $r$ vertices is the following graph: \[\mathsf{P}_r = 
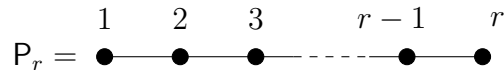
\begin{tikzpicture}
		\foreach \i/\j in {0, 1, 2, 4, 5} {
			\filldraw[black] (\i,0) circle (3pt);
		}
		\draw[-] (0,0)--(2.5,0);
		\draw[-] (5,0)--(3.5,0);
		\draw[dashed] (2.5,0)--(3.5,0);
		\node at (0,0.5) {$1$};
		\node at (1,0.5) {$2$};
		\node at (2,0.5) {$3$};
		\node at (3.8,0.5) {$r-1$};
		\node at (5.2,0.5) {$r$};
	\end{tikzpicture}\] A walk in $\mathsf{P}_r$ is a sequence $(s_0,s_1,\ldots,s_\ell)$ of vertices of $\mathsf{P}_r$ such that $s_t$ and $s_{t+1}$ are adjacent, i.e., $|s_t-s_{t+1}|=1$ for all $t\in [0,\ell]$. The vertices $s_0,s_\ell$ are called the starting and ending vertices of the walk respectively. The number $\ell$ is called the length of the walk. 
\end{defn}

\begin{lem}\label{L:m1paths} Let $p$ be odd, $b$ be a positive integer and $i,j\in [0,\frac{p-3}{2}]$. Then the $(i,j)$-entry of $\M_1^b$ is the number of walks of length $2b$ in the path graph $\mathsf{P}_{p-1}$ starting and ending at $2i+1$ and $2j+1$ respectively. 
\end{lem}
\begin{proof} Suppose that $p=3$. In this case, $\M_1^b=\begin{bmatrix}
    1
\end{bmatrix}$. On the other hand, there is a unique walk of length $2b$ from $1$ to $1$ in $\mathsf{P}_2$. Suppose now that $p\geq 5$. By definition of $\M_1$, we have \[\M_1=\begin{bmatrix}
		1&1&0&\cdots&0&0&0\\ 1&2&1&\cdots& 0&0&0\\ 0&1&2&\cdots&0&0&0\\
		0&0&1&\cdots&0&0&0\\ \vdots&\vdots&\vdots&&\vdots&\vdots&\vdots\\ 0&0&0&\cdots&2&1&0\\ 0&0&0&\cdots&1&2&1\\ 0&0&0&\cdots&0&1&2
	\end{bmatrix}.\]  Let $a_{i,j}(b)$ be the number of walks of length $2b$ starting at the vertex $2i+1$ and ending at the vertex $2j+1$. So we aim to prove that $(\M_1^b)_{i,j}=a_{i,j}(b)$. We proceed via induction on $b$. 
	
	Suppose that $b=1$. Let $i,j\in [0,\frac{p-3}{2}]$. If
$i\neq 0$, then there are exactly two walks of length $2$ starting and ending in $2i+1$, namely $(2i+1, 2i+2, 2i+1)$ and $(2i+1, 2i, 2i+1)$. If $i=0$, then there is exactly one such walk, $(1, 2, 1)$. Thus, it is clear that the diagonal entries $(\M_1)_{i,i}$ are indeed $a_{i,i}(1)$. For the off-diagonal entries, note that if $|i-j|>1$, then there are no walks of length $2$ from $2i+1$ to $2j+1$. If $|i-j|=1$, then there is exactly one such path. This shows that the claim is true for the base case $b=1$.

    Suppose now that the induction hypothesis is correct for some positive integer $b$.  Considering a walk $s=(s_0,\ldots,s_{2b},s_{2b+1},s_{2b+2})$ such that $s_0=2i+1$ and $s_{2b+2}=2j+1$, we have the following (mutually exclusive) cases:
	\begin{enumerate}[(i)]
		\item $s_{2b}=2j+3$ and $s_{2b+1}=2j+2$ (only if $j \neq \frac{p-3}{2}$),
		\item $s_{2b} = 2j+1$ and $s_{2b+1}=2j+2$,
		\item $s_{2b} = 2j+1$ and $s_{2b+1} = 2j$ (only if $j \neq 0$),
		\item $s_{2b}=2j-1$ and $s_{2b+1}=2j$ (only if $j \neq 0$).
	\end{enumerate} There are several cases, depending on whether $j=0$ or $j=\frac{p-3}{2}$, or neither. We consider one case and leave the rest to the reader. If $j\neq 0,\frac{p-3}{2}$, then we have
	\begin{align*}
		(\M_1^{b+1})_{i,j}&=\sum_{t=0}^{\frac{p-3}{2}}(\M_1^b)_{i,t}(\M_1)_{t,j}= (\M_1^b)_{i,j-1}+2(\M_1^b)_{i,j}+(\M_1^b)_{i,j+1} \\
		&= a_{i,j-1}(b) + 2a_{i,j}(b) + a_{i,j+1}(b)= a_{i,j}(b+1),
	\end{align*} where the final equality above follows from the four listed cases. 
    \end{proof}
    
We give a simple example when $p=5$. In turns out that, the entries of $\M_1^b$ are also Fibonacci numbers. 

\begin{eg}
    Suppose that $p=5$. We have $\M_1=\begin{bmatrix}
        1&1\\ 1&2
    \end{bmatrix}$. Let $f_b$'s be the Fibonacci numbers where $f_0=0$ and $f_1=1$. We use the convention $f_{-1}=1$. It is easy to check that \[\M_1^b=\begin{bmatrix}
        f_{2b-1}&f_{2b}\\ f_{2b}& f_{2b+1}
    \end{bmatrix}.\] Together with Lemma \ref{L:m1paths}, we see that the Fibonacci numbers are numbers of certain walks of certain lengths in $\mathsf{P}_4$. 
\end{eg}

We are now ready to state and prove the result for this section. 

\begin{thm}\label{T:D1powers}
    Let $\ell\in [0,\frac{p-3}{2}]$, $b\geq 1$ and $r\in [0,2p-3]$. We have \[(\Omega^r(D_1))^{\otimes 2b}\equiv \bigoplus_{0\leq 2\ell\leq p-3}m_\ell \Omega^{2br}(D_{2\ell})\] where $m_\ell$ is the number of walks of length $2b$ starting and ending in the vertices $2\ell+1$ and 1 respectively in the path graph $\mathsf{P}_{p-1}$. In particular, $[D_1^{\otimes 2b} : D_0]$ is the number of Dyck paths of length $2b$ with restricted height $p-2$ (see \cite[A080934]{OEIS}).
\end{thm}
\begin{proof}
    By Theorem \ref{T:PowersofMk}, $[D_1^{\otimes 2b} : D_{2\ell}]$ is the $(\ell,0)$-entry of $\M_k^b$. By Lemma \ref{L:m1paths}, it is the number of walks of length $2b$ in $\mathsf{P}_{p-1}$ starting in $2\ell+1$ and ending in $1$. When $\ell=0$, the walks of size $2b$ in $\mathsf{P}_{p-1}$ from 1 to 1 are in bijection with the Dyck paths of $2b$ steps with all values less than or equal to $p-2$.
\end{proof}

\section{Wave Sequence}\label{S:Wave}

In the previous section, we examined $\M_1^b$. In this section, we study $\M_2^b$. For this, we require that $p\geq 5$. Again, our aim is to give combinatorial descriptions for the multiplicities of indecomposable summands in the decomposition of $(\Omega^r(D_2))^{\otimes 2b}$. In this case, it turns out that they are numbers appearing in so-called the $h$-wave sequences. The 3-wave sequence has already appeared in the literature \cite[A038196]{OEIS}; otherwise, in general, they seem to be new. 

We first define the $h$-wave sequences. The definition is a bit involved and depends on the parity of $h$.

\begin{defn}\label{def: kwave} Let $h$ be a positive integer. We first define the $h$-wave sequence when $h$ is even. 

Let $h=2\ell$ be a positive even integer, and let $\mathbf{a}=(a_1,\ldots,a_\ell)$ and $\mathbf{b}=(b_1,\ldots,b_\ell)$ be sequences of numbers. We define the $h$-wave sequence $W_h(\mathbf{a},\mathbf{b})$ (with the initials $\mathbf{a}$ and $\mathbf{b}$) as follows: 
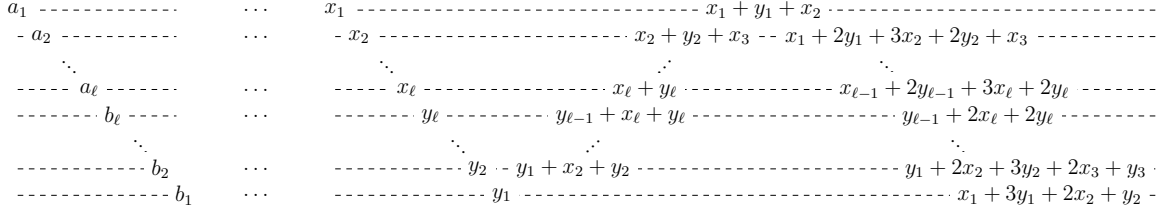
\begin{figure}[H]
\centering
    \[\scalebox{0.7}{\begin{tikzpicture}[every node/.style={fill=white,minimum width = 2ex}]
        \foreach \i in {0,-.5,-1.5,-2,-3,-3.5} {
			\draw[dashed] (0,\i)--(15.5,\i);
			}

        \foreach \i in {0,-.5,-1.5,-2,-3,-3.5} {
			\draw[dashed] (-6,\i)--(-3,\i);
			}

        \foreach \i in {0,-.5,-1.5,-2,-3,-3.5} {
			\node at (-1.5,\i) {$\ldots$};
			}
        
        \begin{scope}[xshift=-6cm]
        \node at (1*.9,-1) {\rotatebox{45}{$\vdots$}};
		\node at (2.5*.9,-2.5) {\rotatebox{45}{$\vdots$}};
		\node at (0*.9,0) {$a_1$};
		\node at (.5*.9,-.5) {$a_2$};
		\node at (1.5*.9,-1.5) {$a_{\ell}$};
		\node at (2*.9,-2) {$b_\ell$};
        \node at (3*.9,-3) {$b_2$};
		\node at (3.5*.9,-3.5) {$b_1$};
        \end{scope}

		\node at (1*.9,-1) {\rotatebox{45}{$\vdots$}};
		\node at (2.5*.9,-2.5) {\rotatebox{45}{$\vdots$}};
		\node at (5.5*.9,-2.5) {\rotatebox{-45}{$\vdots$}};
		\node at (7*.9,-1) {\rotatebox{-45}{$\vdots$}};
		\node at (11.5*.9,-1) {\rotatebox{45}{$\vdots$}};
		\node at (13*.9,-2.5) {\rotatebox{45}{$\vdots$}};
		\node at (0*.9,0) {$x_1$};
		\node at (.5*.9,-.5) {$x_2$};
		\node at (1.5*.9,-1.5) {$x_{\ell}$};
		\node at (2*.9,-2) {$y_\ell$};
			
		\node at (3*.9,-3) {$y_2$};
		\node at (3.5*.9,-3.5) {$y_1$};
		\node at (5*.9,-3) {$y_1+x_2+y_2$};
			
		\node at (6*.9,-2) {$y_{\ell-1}+x_\ell+y_\ell$};
		\node at (6.5*.9,-1.5) {$x_\ell+y_\ell$};
			
		\node at (7.5*.9,-.5) {$x_2+y_2+x_3$};
		\node at (9*.9,0) {$x_1+y_1+x_2$};
		\node at (12*.9,-.5) {$x_1+2y_1+3x_2+2y_2+x_3$};
			
		\node at (13*.9,-1.5) {$x_{\ell-1}+2y_{\ell-1}+3x_\ell+2y_\ell$};
		\node at (13.5*.9,-2) {$y_{\ell-1}+2x_\ell+2y_\ell$};
			
		\node at (14.5*.9,-3) {$y_1+2x_2+3y_2+2x_3+y_3$};
        \node at (15*.9,-3.5) {$x_1+3y_1+2x_2+y_2$};
	\end{tikzpicture}}\]
\caption{$W_h(\mathbf{a},\mathbf{b})$}
\label{Fg:WkEven}
\end{figure}
The numbers in the top line (respectively, bottom line) of the $h$-wave sequence are called the peaks (respectively, troughs). For example, the first peak and trough are $a_1$ and $b_1$ respectively. A period starts from a peak to the next peak, that is, from $x_1$ to $x_1+y_1+x_2$ in Figure \ref{Fg:WkEven}. The first period starts from $a_1$. In general, the $t$th period is the period starting from the $t$th peak. A cycle starts from a peak, to its trough, to the next peak and then ends at the next trough. It is from $x_1$ all the way to $x_1+3y_1+2x_2+y_2$ as Figure \ref{Fg:WkEven} shows. Furthermore, the sequence of every cycle is determined by the beginning of its peak to its next trough which is indicated by $x_1,\ldots,x_\ell,y_\ell,\ldots,y_1$. The cycles continue indefinitely in this manner. In fact, flipping a trough to peak, in this case, \[y_1,y_1+x_2+y_2,\ldots,x_2+y_2+x_3,x_1+y_1+x_2,\] and consider it as a peak to trough by labelling them as $x_1',x_2',\ldots,y_2',y_1'$, using the same formula, we get the next peak to trough. For example, \[y_1'+x_2'+y_2'=(x_1+y_1+x_2)+(y_1+x_2+y_2)+(x_2+y_2+x_3)=x_1+2y_1+3x_2+2y_2+x_3.\] The next case is similar but slightly more involved. 

Suppose now that $h$ is odd and let $h=2\ell+1$. Let $\mathbf{a}=(a_1,\ldots,a_\ell)$ and $\mathbf{b}=(b_1,\ldots,b_\ell)$ as before and also let $c$ be another number. The $h$-wave sequence $W_h(\mathbf{a},\mathbf{b},c)$ (with the initials $\mathbf{a}$, $\mathbf{b}$ and $c$) is defined as follows:
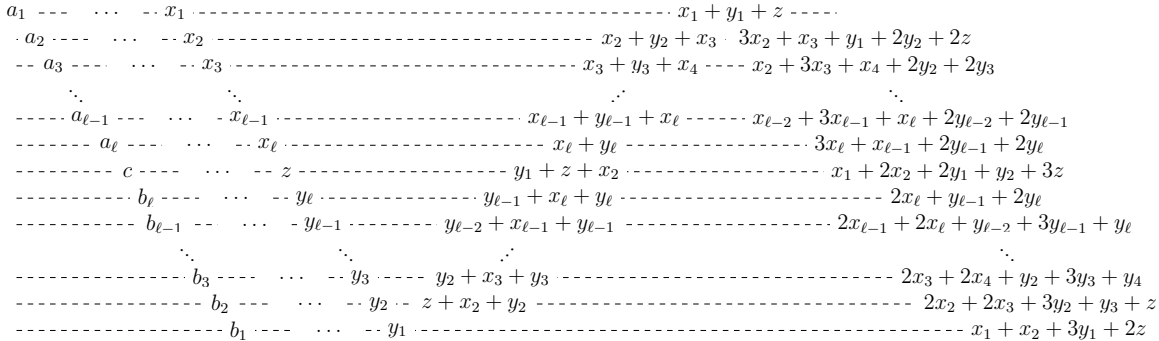
\begin{figure}[H]
\centering
\[\scalebox{0.7}{
    \begin{tikzpicture}[every node/.style={fill=white,minimum width = 2ex}]
        \foreach \i in {0,1,2,4,5,6,7,8,10,11,12} {
			\draw[dashed] (2.4+\i*.35,-\i*.5)--(15.5+\i*.35,-\i*.5);
            \draw[dashed] (0,-\i*.5)--(1+\i*.35,-\i*.5);
            \node at (1.7+\i*.35,-\i*.5) {$\ldots$};
			}
        
        \foreach \label/\i in {a_1/0, a_2/1, a_3/2, \rotatebox{45}{\vdots}/3, a_{\ell-1}/4, a_\ell/5, c/6, b_\ell/7, b_{\ell-1}/8, \rotatebox{45}{\vdots}/9, b_3/10, b_2/11, b_1/12} {\node at (\i*.35,-\i*.5) {$\label$};}
        
        \foreach \label/\i in {x_1/0, x_2/1, x_3/2, \rotatebox{45}{\vdots}/3, x_{\ell-1}/4, x_\ell/5, z/6, y_\ell/7, y_{\ell-1}/8, \rotatebox{45}{\vdots}/9, y_3/10, y_2/11, y_1/12} {\node at (\i*.35+3,-\i*.5) {$\label$};}

        \foreach \label/\i in {x_2+y_2+x_3/1, x_3+y_3+x_4/2, \rotatebox{-45}{\vdots}/3, x_{\ell-1}+y_{\ell-1}+x_\ell/4, x_\ell+y_\ell/5, y_1+z+x_2/6, y_{\ell-1}+x_\ell+y_\ell/7, y_{\ell-2}+x_{\ell-1}+y_{\ell-1}/8, \rotatebox{-45}{\vdots}/9, y_2+x_3+y_3/10, z+x_2+y_2/11} {\node at (12.5-\i*.35,-\i*.5) {$\label$};}

        \node at (13.5,0) {$x_1+y_1+z$}; 

        \foreach \label/\i in {3x_2+x_3+y_1+2y_2+2z/1, x_2+3x_3+x_4+2y_2+2y_3/2, \rotatebox{45}{\vdots}/3, x_{\ell-2}+3x_{\ell-1}+x_\ell+2y_{\ell-2}+2y_{\ell-1}/4, 3x_\ell + x_{\ell-1} + 2y_{\ell-1} + 2y_{\ell}/5, x_1+2x_2+2y_1+y_2+3z/6, 2x_\ell+y_{\ell-1}+2y_\ell/7, 2x_{\ell-1}+2x_\ell+y_{\ell-2}+3y_{\ell-1}+y_\ell/8, \rotatebox{45}{\vdots}/9, 2x_3+2x_4+y_2+3y_3+y_4/10, 2x_2 + 2x_3 + 3y_2 + y_3 + z/11, x_1 + x_2 + 3y_1 + 2z/12} {\node at (\i*.35+15.5,-\i*.5) {$\label$};}
    \end{tikzpicture}}\]
\caption{$W_h(\mathbf{a},\mathbf{b},c)$}
\label{Fg:WkOdd}
\end{figure} The terminologies are similar as in the previous case. Also, the pattern of the $h$-wave sequence is also self-explanatory given in Figure \ref{Fg:WkOdd}. 

Finally, for our interest, we define some specific $h$-wave sequences with specific initials. For these cases, we simplify our notation to $W_h$ as follows:
\[W_h=\left \{\begin{array}{ll}
W_h((1,1,0,\ldots,0),(1,0,0,\ldots,0))& \text{if $h$ is even,}\\
W_h((1,0,\ldots,0),(1,0,\ldots,0),1)& \text{if $h$ is odd.}
\end{array}\right .\]
\end{defn} 

We give examples for the wave sequences $W_3,W_5,W_6$ below, which will, by Lemma \ref{L:m2lemma}, correspond to the primes $7,11,13$ respectively. 

\begin{eg}
    The 3-wave sequence $W_3$ has already appeared in the literature (see \cite[A038196]{OEIS}) and is given as below
\[\scalebox{0.7}{\begin{tikzpicture}[every node/.style={fill=white,minimum width = 2ex}]
    \draw[dashed] (.7,-.5)--(13,-.5);
    \draw[dashed] (0.7,-1)--(13,-1);
    \draw[dashed] (0.7,-1.5)--(13,-1.5);
    \foreach \label/\i/\j in {1/1/1, 1/2/2, 1/3/3, 2/4/2, 3/5/1, 5/6/2, 6/7/3, 11/8/2, 14/9/1, 25/10/2, 31/11/3, 56/12/2, 70/13/1, 126/14/2, 157/15/3, 283/16/2, 353/17/1}
    {\node at (\i*.7,-\j*.5) {$\label$};}
  \end{tikzpicture}}\]
\end{eg}

\begin{eg}\label{Eg:5wave}
The $5$-wave sequence $W_5$ is
\[\scalebox{0.7}{\begin{tikzpicture}[every node/.style={fill=white,minimum width = 2ex}]
	\foreach \i in {0,-.5,-1,-1.5, -2} {
		\draw[dashed] (0,\i)--(15.5,\i);
	}*.7
	\node at (0*.9,0) {$1$};
	\node at (.5*.9,-.5) {$0$};
	\node at (1*.9,-1) {$1$};
	\node at (1.5*.9, -1.5) {$0$};
	\node at (2*.9,-2) {$1$};
	\node at (2.5*.9,-1.5) {$1$};
	\node at (3*.9,-1) {$2$};
	\node at (3.5*.9,-.5) {$0$};
	\node at (4*.9,0) {$3$};
		
	\node at (4.5*.9,-.5) {$3$};
	\node at (5*.9,-1) {$6$};
	\node at (5.5*.9, -1.5) {$1$};
		
	\node at (6*.9,-2) {$6$};

    \begin{scope}[xshift=3.5cm]
	\node at (2.5*.9,-1.5) {$10$};
	\node at (3*.9,-1) {$15$};
	\node at (3.5*.9,-.5) {$4$};
	\node at (4*.9,0) {$15$};
		
	\node at (4.5*.9,-.5) {$29$};
	\node at (5*.9,-1) {$40$};
	\node at (5.5*.9, -1.5) {$14$};
		
	\node at (6*.9,-2) {$36$};
    \end{scope}

    \begin{scope}[xshift=7cm]
	\node at (2.5*.9,-1.5) {$83$};
	\node at (3*.9,-1) {$105$};
	\node at (3.5*.9,-.5) {$43$};
	\node at (4*.9,0) {$91$};
		
	\node at (4.5*.9,-.5) {$231$};
	\node at (5*.9,-1) {$279$};
	\node at (5.5*.9, -1.5) {$126$};
		
	\node at (6*.9,-2) {$232$};
    \end{scope}
		
\end{tikzpicture}}\]
\end{eg}

\begin{eg}\label{Eg:6wave} The $6$-wave sequence $W_6$ is
\[\scalebox{0.7}{\begin{tikzpicture}[every node/.style={fill=white,minimum width = 2ex}]
	\foreach \i in {0,-.5,-1,-1.5,-2,-2.5} {
		\draw[dashed] (0,\i)--(15.5,\i);
	}

	\node at (0*.9,0) {$1$};
	\node at (.5*.9,-.5) {$1$};
	\node at (1*.9,-1) {$0$};
	\node at (1.5*.9,-1.5) {$0$};
		
	\node at (2*.9,-2) {$0$};
	\node at (2.5*.9,-2.5) {$1$};
	\node at (3*.9,-2) {$2$};
		
	\node at (3.5*.9,-1.5) {$0$};
	\node at (4*.9,-1) {$0$};
		
	\node at (4.5*.9,-0.5) {$1$};

    \begin{scope}[xshift=4.5cm]
    \node at (0*.9,0) {$3$};
	\node at (.5*.9,-.5) {$6$};
	\node at (1*.9,-1) {$1$};
	\node at (1.5*.9,-1.5) {$0$};
		
	\node at (2*.9,-2) {$3$};
	\node at (2.5*.9,-2.5) {$6$};
	\node at (3*.9,-2) {$15$};
		
	\node at (3.5*.9,-1.5) {$4$};
	\node at (4*.9,-1) {$1$};
		
	\node at (4.5*.9,-0.5) {$10$};

    \end{scope}

    \begin{scope}[xshift=9cm]
    \node at (0*.9,0) {$15$};
	\node at (.5*.9,-.5) {$40$};
	\node at (1*.9,-1) {$15$};
	\node at (1.5*.9,-1.5) {$5$};
		
	\node at (2*.9,-2) {$29$};
	\node at (2.5*.9,-2.5) {$36$};
	\node at (3*.9,-2) {$105$};
		
	\node at (3.5*.9,-1.5) {$49$};
	\node at (4*.9,-1) {$20$};
		
	\node at (4.5*.9,-0.5) {$84$};
\node at (5*.9,0) {$91$};
    \end{scope}
\end{tikzpicture}}\] 
\end{eg}

Next, we shall examine the powers of the matrix $\M_2$. We leave it to the reader to check that $\M_2$ takes the following general form: 
\begin{equation}\label{Eq:M2}
    \M_2=\begin{bmatrix}
        1&1&1&0&0&0&\cdots&0\\ 
        1&3&2&1&0&0&\cdots&0\\
        1&2&3&2&1&0&\cdots&0\\
        0&1&2&3&2&1&\cdots&0\\
        \vdots&&\ddots&\ddots& \ddots&\ddots&\ddots\\
        0&\cdots&0&1&2&3&2&1\\
        0&\cdots&0&0&1&2&3&2\\
        0&\cdots&0&0&0&1&2&2
    \end{bmatrix}.
\end{equation}

\begin{lem}\label{L:m2lemma} Let $p\geq 5$, $b\geq 1$ and $h=\frac{p-1}{2}$.
\begin{enumerate}[(i)]
	\item Suppose that $h$ is even, i.e., $p\equiv 1\pmod 4$. Let $x_1,x_2,\ldots,x_\ell,y_\ell,\ldots,y_2,y_1$ be the numbers from the peak to the trough in the $b$th period of $W_h$. Then the $(i,0)$-entries (equivalently, the $(0,i)$-entries) of $\M_2^b$ are \[x_1,y_1,x_2,y_2,\ldots,x_\ell,y_\ell\] as $i$ runs from $0$ to $\frac{p-3}{2}$.
	
    \item Suppose that $h$ is odd, i.e., $p\equiv 3\pmod 4$. Let $x_1,x_2,\ldots,x_\ell,z,y_\ell,\ldots,y_2,y_1$ be the numbers from the peak to the trough in the $b$th period of $W_h$. Then the $(i,0)$-entries (equivalently, the $(0,i)$-entries) of $\M_2^b$ are \[x_1,y_1,z,x_2,y_2,\ldots,x_\ell,y_\ell\] as $i$ runs from $0$ to $\frac{p-3}{2}$.
\end{enumerate}
\end{lem}
\begin{proof}
    We assume that $p \geq 13$, as the cases of smaller $p$ are easy to check. We argue by induction on $b$. Suppose that $h$ is even. The $b=1$ case is clear. Suppose that the first column of $\M_2^{b-1}$ is given as $x_1,y_1,x_2,y_2,\ldots,x_\ell,y_\ell$ in the $(b-1)$th period of $W_h$. Then the first column of $\M_2^b$ is \[\M_2\begin{bmatrix} x_1\\ y_1\\ x_2\\ y_2\\ \vdots\\ x_{\ell-1}\\ y_{\ell-1}\\ x_\ell\\ y_\ell
    \end{bmatrix}=\begin{bmatrix} x_1+y_1+x_2\\ x_1+3y_1+2x_2+y_2\\ x_1+2y_1+3x_2+2y_2+x_3\\ y_1+2x_2+3y_2+2x_3+y_3\\ \vdots\\ x_{\ell-2}+2y_{\ell-2}+3x_{\ell-1}+2y_{\ell-1}+x_\ell\\ y_{\ell-2}+2x_{\ell-1}+3y_{\ell-1}+2x_\ell+y_\ell\\ x_{\ell-1}+2y_{\ell-1}+3x_\ell+2y_\ell\\
		y_{\ell-1}+2x_\ell+2y_\ell
	\end{bmatrix}\] The entries in the column matrix on the 
right side are precisely the numbers in the $b$th period of $W_h$ from its peak to its trough. The other case when $h$ is odd is similar and is left to the reader.  
\end{proof}

We illustrate the lemma above using the following two examples. Their wave sequences are given in Examples \ref{Eg:5wave} and \ref{Eg:6wave}. 

\begin{eg}\label{Eg:p11M2}
    Let $p=11$. We have 
\begin{align*}\M_2&= \begin{bmatrix} 1 & 1 & 1 & 0 & 0 \\
1 & 3 & 2 & 1 & 0 \\
1 & 2 & 3 & 2 & 1 \\
0 & 1 & 2 & 3 & 2 \\
0 & 0 & 1 & 2 & 2
\end{bmatrix},& 
\M_2^2&=\begin{bmatrix}
3&6&6&3&1\\
6&15&15&10&4\\
6&15&19&16&9\\
3&10&16&18&12\\
1&4&9&12&9
\end{bmatrix},& \M_2^3&=\begin{bmatrix}
15&36&40&29&14\\
36&91&105&83&43\\
40&105&134&119&69\\
29&83&119&120&76\\
14&43&69&76&51
\end{bmatrix}.
\end{align*} On the other hand, $W_5$ is given as in Example \ref{Eg:5wave}.
\end{eg}

\begin{eg}\label{Eg:p13M2}
    Let $p=13$. We have 
    \begin{align*}
        \M_2&=\begin{bmatrix}
        1&1&1&0&0&0\\ 1&3&2&1&0&0\\
        1&2&3&2&1&0\\
        0&1&2&3&2&1\\
        0&0&1&2&3&2\\
        0&0&0&1&2&2
    \end{bmatrix}, \quad
    \M_2^2=
\begin{bmatrix}
3&6&6&3&1&0\\
6&15&15&10&4&1\\
6&15&19&16&10&4\\
3&10&16&19&16&10\\
1&4&10&16&18&14\\
0&1&4&10&14&14
\end{bmatrix},\\
\M_2^3&=\begin{bmatrix}
15&36&40&29&15&5\\
36&91&105&84&49&21\\
40&105&135&125&89&48\\
29&84&125&141&122&81\\
15&49&89&122&124&94\\
5&21&48&81&94&80
\end{bmatrix}
    \end{align*} On the other hand, $W_6$ is given as in Example \ref{Eg:6wave}.
\end{eg}

We can now state the first result of this section, that is, describing the decomposition of $(\Omega^r(D_2))^{\otimes 2b}$ in terms of the numbers in the wave sequences. 

\begin{thm}
    Let $p\geq 5$, $b\geq 1$, $r\in [0,2p-3]$ and $h=\frac{p-1}{2}$. 
    \begin{enumerate}[(i)] 
        \item Let $h$ be even, i.e., $p\equiv 1\pmod 4$, and let $x_1,x_2,\ldots,x_\ell,y_\ell,\ldots,y_2,y_1$ be the numbers from the peak to the trough in the $b$th period of $W_h$ (so that $h=2\ell$). Then \[(\Omega^r(D_2))^{\otimes 2b}\equiv x_1\Omega^{2br}(D_0)\oplus y_1\Omega^{2br}(D_2)\oplus \cdots\oplus x_{\ell}\Omega^{2br}(D_{4\ell-4})\oplus y_{\ell}\Omega^{2br}(D_{4\ell-2}).\]
        
        \item  Let $h$ be odd, i.e., $p\equiv 3\pmod 4$, and let $x_1,x_2,\ldots,x_\ell,z,y_\ell,\ldots,y_2,y_1$ be the numbers from the peak to the trough in the $b$th period of $W_h$ (so that $h=2\ell+1$). Then 
            \begin{align*}
            (\Omega^r(D_2))^{\otimes 2b}\equiv & \ x_1\Omega^{2br}(D_0)\oplus y_1\Omega^{2br}(D_2)\oplus z\Omega^{2br}(D_4)\oplus x_2\Omega^{2br}(D_6)\oplus y_2\Omega^{2br}(D_8)\\ &\oplus \cdots\oplus 
            \ x_{\ell} \Omega^{2br}(D_{4\ell-2})\oplus y_\ell \Omega^{2br}(D_{4\ell}).
            \end{align*}
    \end{enumerate}
\end{thm}
\begin{proof}
    This follows from Lemma \ref{L:m2lemma} and Theorem \ref{T:PowersofMk}.
\end{proof}

We have successfully identified the multiplicity of the summand $\Omega^{2br}(D_{2i})$ in $(\Omega^r(D_2))^{\otimes 2b}$ as the numbers in certain wave sequences which are also the same as the $0$th column of the matrix $\M_2^n$. However, the remaining entries in $\M_2^n$ also have implicit meaning in the tensor powers of $D_2$. In the sequel, we present a combinatorial description for these entries. 

\begin{defn}\label{def: propagation} Let $x:=(x_0,x_1,\ldots,x_d)$ be a sequence of numbers. For each integer $s\in [0,d]$ and $t\in [0,d]$, the $t$th propagation $z$ in $x$ from the epicenter $x_s$ is the following subsequence of $x$:
	\begin{enumerate}[(i)]
		\item Suppose that $s\leq d-s$.
		\begin{enumerate}[(a)]
			\item if $t\leq s$, then $z=(x_{s-t},x_{s-t+1},\ldots,x_{s+t})$;
			\item if $s+1\leq t\leq d-s$, then $z=(x_{t-s},x_{t-s+1},\ldots,x_{s+t})$;
			\item if $d-s+1\leq t$, then $z=(x_{t-s},x_{t-s+1},\ldots,x_{2d-t-s+1})$.
		\end{enumerate}
		\item Suppose that $d-s< s$.
		\begin{enumerate}[(a)]
			\item if $t\leq d-s$, then $z=(x_{s-t},x_{s-t+1},\ldots,x_{s+t})$;
			\item if $d-s+1\leq t\leq s$, then $z=(x_{s-t},x_{s-t+1},\ldots,x_{2d-t-s+1})$;
			\item if $s+1\leq t$, then $z=(x_{t-s},x_{t-s+1},\ldots,x_{2d-t-s+1})$.
		\end{enumerate}
	\end{enumerate} The total value of a $t$th propagation is the sum of the values in that subsequence. 
\end{defn}

We illustrate the definition as follows. The $t$th propagation describes a wave spreading outward from $x_s$. At first, the wave expands symmetrically to the left and right, so the propagation contains all terms between $x_{s-t}$ and $x_{s+t}$. When the wave reaches one of the boundaries, $x_0$ or $x_d$, it does not simply stop there. Instead, the wave is reflected back into the sequence. Thus the propagation continues by “bouncing” off the boundary. As the reflected part overlaps with the still-expanding part, some entries are effectively canceled out, producing a kind of destructive interference. Because of this interference, the visible propagation interval can begin to shrink after enough reflections occur. The reflection at the left boundary is immediate while the reflection at the right boundary is delayed.

\begin{eg}\label{Eg:p11Prop} Let $x=(15,36,40,29,14)$ the first row of our matrix $\M_2^3$ in Example \ref{Eg:p11M2}. For $t\in [0,4]$, the $t$th propagations $z$ in $x$ from the epicenter $x_3=29$ and their total values $T$ are given below:
\[\begin{tabular}{ccc}
  \hline 
  $t$&$z$&$T$\\
  \hline 
  0&(29)&29\\
  1&(40,29,14)&83\\
  2&(36,40,29,14)&119\\
  3&(15,36,40,29)&120\\
  4&(36,40) &76\\
  \hline
\end{tabular}\]
\end{eg}

The total values of propagations have the following symmetry property. 

\begin{lem}\label{L:propsymm}
  Let $x=(x_0,\ldots,x_d)$ and $s,t\in[0,d]$. Then the total values of the $t$th propagation in $x$ from the epicenter $x_s$ and the $s$th propagation in $x$ from the epicenter $x_t$ are the same.
\end{lem}
\begin{proof}
  Let $z$ (respectively, $z'$) be the $t$th (respectively, $s$th) propagation in $x$ from the epicenter $x_s$ (respectively, $x_t$). We check the case when the pair $(s,t)$ satisfies hypothesis (i) in Definition \ref{def: propagation} and leave the other to the reader. 
  \begin{enumerate}[(a)]
    \item Suppose that $0\leq t\leq s\leq d-s$. We have $z=(x_{s-t},x_{s-t+1},\ldots,x_{s+t})$. If $s=t$, clearly, $z=z'$. Suppose that $t<s$. Since $t\leq d-t$ and $t+1\leq s$, $z'=(x_{s-t},x_{s-t+1},\ldots,x_{s+t})$. So $z=z'$.
    \item Suppose that $s<t\leq d-s$. We have $z=(x_{t-s},x_{t-s+1},\ldots,x_{s+t})$. If $t\leq d-t$, since $s\leq t$, then $z'=(x_{t-s},x_{t-s+1},\ldots,x_{t+s})=z$. If $d-t<t$, since $s\leq d-t$, we also get $z'=z$.
    \item Suppose that $s\leq d-s<t$. We have $z= (x_{t-s},x_{t-s+1},\ldots,x_{2d-t-s+1})$. On the other hand, we have both $d-t<t$ and $d-t+1\leq s\leq t$. So $z'=(x_{t-s},x_{t-s+1},\ldots,x_{2d-s-t+1})=z$. 
  \end{enumerate} 
\end{proof}

We can now describe each of the entries of $\M_2^b$ in terms of the top row of the matrix. The complete proof requires us to consider many cases and is rather long. For brevity, we only work out a few cases and refer the reader to the first author's PhD thesis \cite{KuaThesis} for the remaining ones. 

\begin{prop}\label{P:propsymm}
    Let $x=(x_0,x_1,\ldots,x_{\frac{p-3}{2}})$ be the top row of $\M_2^b$. Then, for each $i,j\in [0,\frac{p-3}{2}]$, the $(i,j)$-entry of $\M_2^b$ is the total value of the $i$th propagation in $x$ from the epicenter $x_j$.
\end{prop}
\begin{proof}
    Since $\M_2$ is symmetric, $x$ is also the first column of $\M_2^b$. Let $h=\frac{p-3}{2}$. We proceed via induction on $b$. For $i,j\in [0,h]$, let $z$ denote the $i$th propagation in $x$ from the epicenter $x_j$. 
    
    For $b=1$, we have $x=(1,1,1,0,\ldots,0)$. If $i=0=j$, then $z=(1)$. If $h\neq i=j\neq 0$, then $z=(1,1,1,0,\ldots,0)$. If $i=j=h$, then $z=(1,1)$. These justify the diagonal entries of $\M_2$ in Equation (\ref{Eq:M2}). The remaining cases are left to the reader. 

    Suppose now that $b\geq 2$. Let $x'=(x_0',x_1',\ldots,x_h')$ be the first row of $\M_2^{b-1}$. Therefore, 
    \begin{align}\label{Eq:xx'}
        x_0&=x_0'+x_1'+x_2',\notag\\
        x_1&=x_0'+3x_1'+2x_2'+x_3',\notag\\
        x_c&=x_{c-2}'+2x_{c-1}'+ 3x_c'+2x_{c+1}'+x_{c+2}',\notag\\
        x_{h-1}&=x_{h-3}'+2x_{h-2}'+3x_{h-1}'+2x_h',\notag\\
        x_h&=x_{h-2}'+2x_{h-1}'+2x_h',
    \end{align} where $c\in [2,h-2]$. On the other hand, we calculate \[\mu_{i,j}=(\M_2^b)_{i,j}=(\M_2\M_2^{b-1})_{i,j}= \sum_{c=0}^{h}(\M_2)_{i,c}(\M_2^{b-1})_{c,j}.\] To prove the result, there are many cases to check. But, since $\M_2$ is symmetric, we have $\mu_{i,j}=\mu_{j,i}$ and, by Lemma \ref{L:propsymm}, we can assume that $i\leq j$. However, we first deal with the case when $j=0,1$ and $i$ is arbitrary because it is special. 
    
    If $j=0$, for each $i\in [0,h]$, we have $z=(x_i)$. On the other hand, $(\M_2^b)_{i,0}=x_i$. Suppose that $j=1$. There are 3 subcases.
    \begin{enumerate}[(i)]
      \item Suppose that $i\in \{0,1,2,3\}$. If $i=0$, this is similar to the $j=0$ case because $\M_2$ is symmetric. Now let $i=1$. We have $z=(x_0,x_1,x_2)$. On the other hand, 
          \begin{align*} 
          \mu_{1,1}&= (\M_2^{b-1})_{0,1}+3(\M_2^{b-1})_{1,1}+ 2(\M_2^{b-1})_{2,1}+(\M_2^{b-1})_{3,1}\\
          &=x_1'+3(x_0'+x_1'+x_2')+2(x_1'+x_2'+x_3')+ (x_2'+x_3'+x_4')\\
          &=x_0+x_1+x_2.
          \end{align*} The cases $i=2,3$ are left to the reader. 
      \item Suppose that $4\leq i\leq h-1$. Assume first that $i\leq h-3$. We have $z=(x_{i-1},x_i,x_{i+1})$. On the other hand, 
          \begin{align*} 
          \mu_{i,1}&=
          (\M_2^{b-1})_{i-2,1}+2(\M_2^{b-1})_{i-1,1}+ 3(\M_2^{b-1})_{i,1}+2(\M_2^{b-1})_{i+1,1}+ (\M_2^{b-1})_{i+2,1}\\
          &=(x_{i-3}'+x_{i-2}'+x_{i-1}')+ 2(x_{i-2}'+x_{i-1}'+x_i')+ 3(x_{i-1}'+x_i'+x_{i+1}')+ \\ &\ \ \ \ 2(x_i'+x_{i+1}'+x_{i+2}')+ (x_{i+1}'+x_{i+2}'+x_{i+3}')\\
          &=x_{i-1}+x_i+x_{i+1}
          \end{align*} The cases for $i=h-2,h-1$ are similar with some slight modification. 
      \item Suppose that $i=h$. We have $z=(x_{h-1},x_h)$. On the other hand, 
          \begin{align*}
            \mu_{h,1}&=(\M_2^{b-1})_{h-2,1}+ 2(\M_2^{b-1})_{h-1,1}+ 2(\M_2^{b-1})_{h,1}\\
            &=(x_{h-3}'+x_{h-2}'+x_{h-1}')+ 2(x_{h-2}'+x_{h-1}'+x_h')+2(x_{h-1}'+x_h')\\
            &=x_{h-1}+x_h.
          \end{align*}
    \end{enumerate} The case when $j=0,1$ and $i$ is arbitrary is done.
    
    We now assume that $i\leq j$ and $2\leq j\leq h-2$. If $i=0,1$, since $\mu_{i,j}=\mu_{j,i}$, these cases have been dealt with. So, we can focus on the case when $2\leq i\leq j\leq h-2$. 
    \begin{enumerate}[(i)]
	  \item Suppose that $2\leq i\leq j-2$. We have 
        $z=(x_{j-i},x_{j-i+1},\ldots,x_{j+i})$. On the other hand, we have 
        \begin{align*} 
        \mu_{i,j}&=(\M_2^{b-1})_{i-2,j}+2(\M_2^{b-1})_{i-1,j}+ 3(\M_2^{b-1})_{i,j}+ 2(\M_2^{b-1})_{i+1,j}+ (\M_2^{b-1})_{i+2,j}\\
        &=(x_{j-i+2}'+x_{j-i+3}'+\cdots+x_{j+i-2}')+2(x_{j-i+1}'+ x_{j-i+2}'+\cdots+x_{j+i-1}')\\
        &\ \ \ \ +3(x_{j-i}'+x_{j-i+1}'+\cdots+x_{j+i}')+2(x_{j-i-1}'+ x_{j-i}'+\cdots+x_{j+i+1}')\\
        &\ \ \ \ +(x_{j-i-2}'+x_{j-i-1}'+\cdots+x_{j+i+2}')\\
        &=x_{j-i}+x_{j-(i-1)}+\cdots+x_{j+(i-1)}+x_{j+i},
        \end{align*} 
        where the penultimate equation follows using induction hypothesis and the last equation can be proved by using induction on $i$ (the base case is $i=2$) and the system of linear equations in (\ref{Eq:xx'}). 
      \item Suppose that $2\leq i=j-1$. Suppose first that $j\leq h-j$. We have $z=(x_1,x_2,\ldots,x_{2j-1})$. If $j<h-j$, then we have 
          \begin{align*}
            \mu_{j-1,j}&=(\M_2^{b-1})_{j-3,j}+ 2(\M_2^{b-1})_{j-2,j}+ 3(\M_2^{b-1})_{j-1,j}+ 2(\M_2^{b-1})_{j,j}+(\M_2^{b-1})_{j+1,j}\\
            &=(x_3'+x_4'+\cdots+x_{2j-3}')+ 2(x_2'+x_3'+\cdots+x_{2j-2}')+ 3(x_1'+x_2'+\cdots+x_{2j-1}')\\&\ \ \ \  
            + 2(x_0'+x_1'+\cdots+x_{2j}') + (x_1'+x_2'+\cdots+x_{2j+1}')\\
            &=x_1+(x_4'+\cdots+x_{2j-3}')+ 2(x_3'+\cdots+x_{2j-2}')+ 3(x_2'+\cdots+x_{2j-1}')\\&\ \ \ \ 
            + 2(x_1'+\cdots+x_{2j}') + (x_0'+x_1'+x_2'+\cdots+x_{2j+1}')\\
            &=x_1+x_2+\cdots+x_{2j-1},
          \end{align*} where the last equality follows the similar argument as in part (i). If $j=h-j$, then similar calculation shows
          \begin{align*}
            \mu_{j-1,j}&=(x_3'+x_4'+\cdots+x_{h-3}')+ 2(x_2'+x_3'+\cdots+x_{h-2}') + 3(x_1'+x_2'+\cdots+x_{h-1}')\\&\ \ \ \ 
            + 2(x_0'+x_1'+\cdots+x_{h}') + (x_1'+x_2'+\cdots+x_h')\\
            &=x_1+(x_4'+x_5'+\cdots+x_{h-4}')+ 2(x_3'+x_4'+\cdots+x_{h-3}') + 3(x_2'+x_3'+\cdots+x_{h-2}')\\&\ \ \ \ 
            + 2(x_1'+x_2'+\cdots+x_{h-1}') + (x_0'+x_1'+\cdots+x_h')+x_{h-1}\\
            &=x_1+x_2+\cdots+x_{h-2}+x_{h-1}.
          \end{align*} Suppose now that $h-j<j$. We leave the case $h-j=j-1$ to the reader. Suppose that $h-j+1\leq j-1$. We have $z=(x_1,x_2,\ldots,x_{2h-2j+2})$. We deal with the case when $2j=h+2$ case and leave out the rest. Similar calculation shows
          \begin{align*}
            \mu_{j-1,j}&=(x_3'+x_4'+\cdots+x_{2j-3}')+ 2(x_2'+x_3'+\cdots+x_{2j-2}') + 3(x_1'+x_2'+\cdots+x_{2h-2j+2}')
            \\&\ \ \ \   + 2(x_0'+x_1'+\cdots+x_{2h-2j+1}') + (x_1'+x_2'+\cdots+x_{2h-2j}')\\
            &=(x_3'+x_4'+\cdots+x_{h-1}')+ 2(x_2'+x_3'+\cdots+x_{h}') + 3(x_1'+x_2'+\cdots+x_{h}')
            \\&\ \ \ \   + 2(x_0'+x_1'+\cdots+x_{h-1}') + (x_1'+x_2'+\cdots+x_{h-2}')\\
            &=(x_4'+\cdots+x_{h-4}')+ 2(x_3'+\cdots+x_{h-3}') + 3(x_2'+\cdots+x_{h-2}')
            \\&\ \ \ \   + 2(x_1'+\cdots+x_{h-1}') + (x_0'+x_1'+x_2'+\cdots+x_{h-1}'+x_h')\\
            &=x_1+x_2+\cdots+x_{h-1}+x_{h}.
          \end{align*} 
      \end{enumerate}
\end{proof}

The following example illustrates our Proposition \ref{P:propsymm}.

\begin{eg}
  Let $p=11$. The total values of the $t$th propagations from the epicenter $x_3$ presented in Example \ref{Eg:p11Prop} are precisely the entries in the `third' column (or row) of $\M_2^3$ in Example \ref{Eg:p11M2}. Below, we also compute the $t$th propagation $z$ in $x$ from the epicenter $x_2$ and their values. The values are then precisely the `second' column (or row) of $\M_2^3$. 
  \[\begin{tabular}{ccc}
  \hline 
  $t$&$z$&$T$\\
  \hline 
  0&(40)&40\\
  1&(36,40,29)&105\\
  2&(15,36,40,29,14)&134\\
  3&(36,40,29,14)&119\\
  4&(40,29) &69\\
  \hline
\end{tabular}\]
\end{eg}

\section{Stability when $p$ tends to infinity}\label{S:Stability}

In this section, we consider the stability question when the underlying prime goes to infinity. Since the prime $p$ varies, we highlight it for the simple module $D_k$ by using the notation $D_{k,p}$. Similarly, we use $\M_{k,p}$ and $\Rect_p(i,j)$ for $\M_k$ and $\Rect(i,j)$ respectively. The main result for this section is Theorem \ref{T:pToinfinity}.

To prove our result, we need the following lemmas comparing the powers of the matrices $\M_{k,p}$ and $\M_{k,q}$ whenever $q\geq p$. 

\begin{lem}\label{L:D Stability}
  Let $k\in\NN_0$. If $p$ is a prime such that $k=\min\{k,p-2-k\}$ and $\ell \in [0,\frac{p-3}{2}-k]$, then, for all prime $q\geq p$, $i\in [0,k]$ and $j \in [0,\frac{p-3}{2}]$, we have \[[D_{2i,q}\otimes D_{2\ell,q}:D_{2j,q}]=[D_{2i,p}\otimes D_{2\ell,p}:D_{2j,p}].\] 
\end{lem}
\begin{proof}
  Recall that $[D_{2i}\otimes D_{2\ell}:D_{2j}]\neq 0$ if and only if $2j \in \Rect(2i,2\ell)$. Let $p$ be such a prime. When $\ell \in [0,\frac{p-3}{2}-k]$ and $i\in [0,k]$, we have $2i\leq 2k\leq p-3-2\ell$ and therefore $2i\leq q-3-2\ell$. In this case, the $2i$th layers of the $2\ell$-rectangles for $p$ and $q$ coincide. It can be depicted as below where the darker rectangle is for $p$ and the other for $q$.
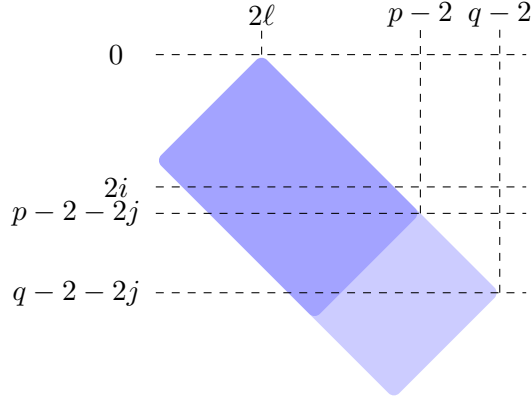
\[\small{
\begin{tikzpicture}[xscale=.35,yscale=.35]
\fill[blue,rounded corners=1mm,opacity=0.2] (0,0)--(4,4)--(10,-2)--(6,-6)--(0,0)--(4,4)--cycle;
\fill[blue,rounded corners=1mm,opacity=0.2] (0,0)--(4,4)--(13,-5)--(9,-9)--(0,0)--(4,4)--cycle;

\draw[dashed] (4,4)--(4,5);
\node at (4,5.5) {$2\ell$};
\draw[dashed] (10,-2)--(10,5);
\node at (10,5.5) {$p-2$};
\draw[dashed] (13,-5)--(13,5);
\node at (13,5.5) {$q-2$};
\draw[dashed] (0,-2)--(14,-2);
\node at (-3,-2) {$p-2-2j$};
%\draw[decorate, decoration={brace,mirror,amplitude=10pt,raise = 0pt},thick] (0,4)--(0,-1);
\draw[dashed] (0,-1)--(14,-1);
\node at (-1.5,-1) {$2i$};
\draw[dashed] (0,4)--(14,4);
\node at (-1.5,4) {$0$};
\draw[dashed] (0,-5)--(14,-5);
\node at (-3,-5) {$q-2-2j$};
\end{tikzpicture}}\]
\end{proof}

\begin{lem}\label{L:MpMqn=1}
  Let $k\in\NN_0$ and $p$ be a large enough prime such that $k=\min\{k,p-2-k\}$ and $\frac{p-3}{2}-k\geq 0$. For $\ell\in [0,\frac{p-3}{2}-k]$, we have the following statements. 
  \begin{enumerate}[(i)]
    \item We have $(\M_{k,p})_{\ell,\ell+k}\neq 0$ but $(\M_{k,p})_{\ell,j}= 0$ for all $\ell+k<j\leq \frac{p-3}{2}$, that is, the rightmost nonzero entry in the $\ell$th row of $\M_{k,p}$ is $\ell+k$. 
    \item For any prime $q\geq p$, we have \[(\M_{k,q})_{\ell,j}=\left \{\begin{array}{ll} (\M_{k,p})_{\ell,j}& \text{if $j\leq \frac{p-3}{2}$,}\\ 0&\text{otherwise.}\end{array}\right .\]
  \end{enumerate}
\end{lem}
\begin{proof}
  Recall that $(\M_k)_{\ell,j}=\sum_{i=0}^m[D_{2i}\otimes D_{2j}:D_{2\ell}]$. For part (i), when $j=\ell+k$, take $i=m=k$, then $[D_{2k,p}\otimes D_{2(\ell+k),p}:D_{2\ell,p}]=1\neq 0$ because $2\ell\in \Rect_p(2k,2(\ell+k))$. So $(\M_{k,p})_{\ell,\ell+k}\neq 0$. Suppose now $j>\ell+k$. For any $i\in [0,m]$, $2\ell\not\in \Rect_p(2i,2j)$ because $2\ell<2j-2k=2j-2m\leq 2j-2i$. So $[D_{2i,p}\otimes D_{2j,p}:D_{2\ell,p}]=0$. Hence $(\M_{k,p})_{\ell,j}=0$.
  
  We now prove part (ii). Let $q \geq p$ and $j \leq \frac{p-3}{2}$. By definition, Lemma \ref{L:MkSymm} and Lemma \ref{L:D Stability},
  \begin{align*}
  	(\M_{k,q})_{\ell,j} &= (\M_{k,q})_{j,\ell}= \sum_{i=0}^k [D_{2i,q}\otimes D_{2\ell,q}:D_{2j,q}]  = \sum_{i=0}^k [D_{2i,p}\otimes D_{2\ell,p}:D_{2j,p}]\\
   &= (\M_{k,p})_{j,\ell} = (\M_{k,p})_{\ell,j}. \end{align*}
  For $j \in [\frac{p-3}{2}+1, \frac{q-3}{2}]$, we have $j > \frac{p-3}{2} \geq \ell+k$. By part (i), $(\M_{k,q})_{\ell,j}=0$.
\end{proof}

To compare the tensor powers of $D_{k,p}$ and $D_{k,q}$, we need to compare the powers of $\M_{k,p}$ and $\M_{k,q}$. We get the following generalised version of Lemma \ref{L:MpMqn=1}:

\begin{lem}\label{L:MpMqGeneral}
	 Let $k,b\in\NN_0$ with $b\geq 1$, $p$ be a large enough prime such that $k=\min\{k,p-2-k\}$ and $\frac{p-3}{2}-bk\geq 0$ and let $q\geq p$. Furthermore, let $\ell \in [0,\frac{p-3}{2} - bk]$.
	\begin{enumerate}[(i)]
		\item The rightmost nonzero entry in the $\ell$th row of $\M_{k,p}^b$ is in $\ell+bk$.
	 \item We have \[(\M_{k,q}^b)_{\ell,j}=\left \{\begin{array}{ll} (\M_{k,p}^b)_{\ell,j}& \text{if $j\leq \frac{p-3}{2}$,}\\ 0&\text{otherwise.}\end{array}\right .\]
	 \end{enumerate}
\end{lem}
\begin{proof}
	\begin{enumerate}[(i)]
		\item The base case $b=1$ is true by part (i) of Lemma \ref{L:MpMqn=1}. We write
		\[(\M_{k,p}^b)_{\ell,j} = \sum_{t=0}^{\frac{p-3}{2}} (\M_{k,p})_{\ell,t} (\M_{k,p}^{b-1})_{t,j}.\]
		For this to be nonzero, we require $t$ such that $t \leq \ell + k$ and $j \leq t + (b-1)k$ (by induction hypothesis). Combining these inequalities, we get $j \leq \ell + bk$.	The rightmost nonzero entry in the $\ell$th row of $\M_{k,p}$ is in $\ell+k$. By induction hypothesis, the rightmost nonzero entry in the $(\ell+k)$th row of $\M_{k,p}^{b-1}$ is in $\ell+k+(b-1)k = \ell+bk$. Thus, the summand $(\M_{k,p})_{\ell,\ell+k} (\M_{k,p}^{b-1})_{\ell+k,\ell+bk}$ is nonzero. Since the matrices have nonnegative integer entries, this implies the entry $(\M_{k,p}^b)_{\ell,\ell+bk}$ is not zero.
		\item Fix $k,b\in\NN_0$ and $q\geq p$ as in the hypothesis. Let $P(r)$ be the assertion that for any $\ell\in [0,\frac{p-3}{2}-rk]$, we have \[(\M_{k,q}^r)_{\ell,j}=\left \{\begin{array}{ll} (\M_{k,p}^r)_{\ell,j}& \text{if $j\leq \frac{p-3}{2}$,}\\ 0&\text{otherwise.}\end{array}\right .\] We argue by induction on $r\leq b$. The base case $r=1$ is given in Lemma \ref{L:MpMqn=1}(ii). Assume that $P(r)$ is correct and $2 \leq r+1\leq b$. Let $\ell\in [0,\frac{p-3}{2}-(r+1)k]$. We have 
  \begin{align*}\label{Eq:IndMq}
  (\M_{k,q}^{r+1})_{\ell,j}=\sum_{x=0}^{\frac{q-3}{2}} (\M_{k,q})_{\ell,x}(\M_{k,q}^r)_{x,j}&= \sum_{x=0}^{\ell+k} (\M_{k,q})_{\ell,x}(\M_{k,q}^r)_{x,j}\\
  &=\left \{\begin{array}{ll} \sum_{x=0}^{\ell+k} (\M_{k,p})_{\ell,x}(\M_{k,p}^r)_{x,j}&\text{if $j\leq\frac{p-3}{2}$,}\\ 0&\text{otherwise,}\end{array}\right .
  \end{align*} where the second equality follows from Lemma \ref{L:MpMqn=1}(i) because $\ell+k\leq \frac{p-3}{2}-k\leq \frac{q-3}{2}$ and the third equality follows from induction hypothesis and Lemma \ref{L:MpMqn=1}(ii) because $x\leq \ell+k\leq \frac{p-3}{2}-rk\leq \frac{p-3}{2}$. We focus only on the case $j\leq \frac{p-3}{2}$. Similarly, since $\ell+k\leq\frac{p-3}{2}$ and using Lemma \ref{L:MpMqn=1}(i), we have \[(\M_{k,p}^{r+1})_{\ell,j}=\sum_{x=0}^{\frac{p-3}{2}} (\M_{k,p})_{\ell,x}(\M_{k,p}^r)_{x,j}= \sum_{x=0}^{\ell+k} (\M_{k,p})_{\ell,x}(\M_{k,p}^r)_{x,j}.\] The proof is now complete. 
  \end{enumerate}
\end{proof}

We can now state the stability result in this section. 

\begin{thm}\label{T:pToinfinity}
  Let $k,n\in\NN_0$ with $n\geq 1$. There is a large enough prime $p$ such that, for any $r \in [0,2p-3]$ and prime $q\geq p$ and $t\in[0,q-2]$, we have \[[(\Omega^r(D_{k,q}))^{\otimes n}:\Omega^{nr}(D_{t,q})]=\left \{\begin{array}{ll} [(\Omega^r(D_{k,p}))^{\otimes n}:\Omega^{nr}(D_{t,p})]& \text{if $t\leq p-2$,}\\ 0&\text{if $t>p-2$.}\end{array}\right .\] In particular, for $k,t,n\in\NN_0$, the following limit exists:
  \[\lim_{\substack{p \to \infty \\ \text{$p$ is prime}}} [D_{k,p}^{\otimes n} : D_{t,p}].\]
\end{thm}
\begin{proof}
	We may assume $r=0$. The general case can be obtained by applying $\Omega^{nr}$.

  We first consider $n=2b$. Choose $p$ large enough such that it satisfies the hypothesis for Lemma \ref{L:MpMqGeneral} with respect to $k$ and $b$, that is $k=\min\{k,p-2-k\}$ and $\frac{p-3}{2}-bk\geq 0$ so $p \geq 2bk+3 = nk + 3$. Let $t\in[0,q-2]$, $y=[D_{k,q}^{\otimes n}:D_{t,q}]$ and, if $t\leq p-2$, $x=[D_{k,p}^{\otimes n}:D_{t,p}]$. If $t$ is odd, then $x=0=y$ by Proposition \ref{P:reduction}. Suppose now that $t=2j\leq p-2$ is even. We have 
  \begin{align*}
    x=(\M_{k,p}^b)_{j,0}=(\M_{k,p}^b)_{0,j}= (\M_{k,q}^b)_{0,j}=(\M_{k,q}^b)_{j,0}=y,
  \end{align*} where the first and last equalities follow from Theorem \ref{T:PowersofMk}, the second and penultimate equalities follow from Lemma \ref{L:MkSymm} and the third equality follows from Lemma \ref{L:MpMqGeneral}(ii). For $t=2j>p-2$, we have $y=(\M_{k,q}^b)_{j,0}=0$ by Lemma \ref{L:MpMqGeneral}(ii) again. 
  
  Suppose now that $n=2b+1$. Again, choose $p \geq nk +3$. By the previous paragraph, we have  $D_{k,p}^{\otimes 2b}\equiv\bigoplus_{0\leq 2i\leq p-3}a_{2i}D_{2i,p}$ and $D_{k,q}^{\otimes 2b}\equiv\bigoplus_{0\leq 2i\leq q-3}b_{2i}D_{2i,q}$ where $a_{2i}=b_{2i}$ if $2i\leq p-3$ and $b_{2i}=0$ if $2i>p-3$. Let $t\leq p-2$. We have  
  \begin{align*}
    [D_{k,p}^{\otimes n}:D_{t,p}]=\sum_{t\in\Rect_p(k,2i)}a_{2i}= \sum_{t\in\Rect_p(k,2i)}b_{2i},
  \end{align*} where the first equality is obtained using Proposition \ref{P:reduction}. On the other hand, \[[D_{k,q}^{\otimes n}:D_{t,q}]=\sum_{t\in\Rect_q(k,2i)}b_{2i}.\] To show our result, notice that $b_{2i}$ is the $i$th entry in the $0$th row of $\M_{k,q}^b$, so Lemma \ref{L:MpMqGeneral}(i) asserts that nonzero $b_{2i}$ only occurs for $i\leq bk$. If $t \in \Rect_q(k,2i)$, then $t \leq 2i + k \leq 2bk + k = (n-1)k + k = nk \leq p-3$. Thus $\Rect_p(k,2i)$ and $\Rect_q(k,2i)$ coincide when $b_{2i}$ is nonzero. Combining this with the fact that $b_{2i} = 0$ for all $2i > p-3$. So $\sum_{t\in\Rect_p(k,2i)}b_{2i}= \sum_{t\in\Rect_q(k,2i)}b_{2i}$.

   Suppose that $t>p-2$ and $k$ is even. By Definition \ref{D:Mvnk}, note that $v_1^{(k)}$ has $1$ in the $\frac{k}{2}$th entry and $0$ elsewhere. Since $n = 2b+1$, by Lemma \ref{L:vnk}, we get $v_{n}^{(k)} = v_{2b+1}^{(k)} = \M_{k,q}^b v_1^{(k)}$, so $v_n^{(k)}$ is the $\frac{k}{2}$th column of $\M_{k,q}^b$.  This means that  the decomposition of $D_{k,q}^{\otimes n}$ may be read off the $\frac{k}{2}$th column of $\M_{k,q}^b$. By symmetry of $\M_{k,q}$ (Lemma \ref{L:MkSymm}), this is equivalent to reading off the $\frac{k}{2}$th row. That is, $[D_{k,q}^{\otimes n} : D_{t,q}] = (\M_{k,q}^b)_{k/2, t/2}$ where $t = 2j > p-2$. By assumption on $p$, we have $p-3 \geq nk = (n-1)k + k$, hence \[\frac{p-3}{2} - \frac{n-1}{2} k \geq \frac{k}{2}.\] Note that $b = \frac{n-1}{2}$, so $\frac{k}{2} \in [0, \frac{p-3}{2} - bk]$. Furthermore, $t > p-2$, so $j > \frac{p-3}{2}$. Applying Lemma \ref{L:MpMqGeneral}(ii) yields $[D_{k,q}^{\otimes n} : D_{t,q}] = 0$.

  Suppose now that $k$ is odd. For $t\in\Rect_q(k,2i)$ and $b_{2i}\neq 0$, we require that $t\leq 2i+k$ and by Lemma \ref{L:MpMqGeneral}(i), we also need $i \leq bk$. This yields $2i \leq 2bk  = (n-1)k$ and $t \leq  (n-1)k+ k = nk$. But $p \geq nk +3$ so this is a contradiction, since $t > p-2$ implies that $t > nk+3-2 = nk+1 > nk$.
\end{proof}

We end this paper with an example illustrating Theorem \ref{T:pToinfinity}. 

\begin{eg}
  Consider $D_{2,p}^{\otimes 6}$, that is, $k=2$, $n=6=2b$. For a prime $p$ to meet the hypothesis of Lemma \ref{L:MpMqGeneral}, we require that $p\geq 17$. Let $p=17$ and $q=19$. The lemma asserts that, for all $\ell\in [0,1]$, we have $(\M_{2,19}^3)_{\ell,j}$ is $(\M_{2,17}^3)_{\ell,j}$ if $j\in [0,7]$ and 0 otherwise. Indeed, the matrices $\M_{2,17}^3$ and $\M_{2,19}^3$ are given as below.
  \begin{align*}
    &\M_{2,17}^3=\small{\begin{bmatrix}
    		15 & 36 & 40 & 29 & 15 & 5 & 1 & 0 \\
    		36 & 91 & 105 & 84 & 49 & 21 & 6 & 1 \\
    		40 & 105 & 135 & 125 & 90 & 50 & 21 & 6 \\
    		29 & 84 & 125 & 141 & 126 & 90 & 50 & 20 \\
    		15 & 49 & 90 & 126 & 141 & 126 & 89 & 44 \\
    		5 & 21 & 50 & 90 & 126 & 140 & 120 & 69 \\
    		1 & 6 & 21 & 50 & 89 & 120 & 120 & 76 \\
    		0 & 1 & 6 & 20 & 44 & 69 & 76 & 51
    \end{bmatrix}}, \\
 &\M_{2,19}^3=\small{\begin{bmatrix}
 		15 & 36 & 40 & 29 & 15 & 5 & 1 & 0 & 0 \\
 		36 & 91 & 105 & 84 & 49 & 21 & 6 & 1 & 0 \\
 		40 & 105 & 135 & 125 & 90 & 50 & 21 & 6 & 1 \\
 		29 & 84 & 125 & 141 & 126 & 90 & 50 & 21 & 6 \\
 		15 & 49 & 90 & 126 & 141 & 126 & 90 & 50 & 20 \\
 		5 & 21 & 50 & 90 & 126 & 141 & 126 & 89 & 44 \\
 		1 & 6 & 21 & 50 & 90 & 126 & 140 & 120 & 69 \\
 		0 & 1 & 6 & 21 & 50 & 89 & 120 & 120 & 76 \\
 		0 & 0 & 1 & 6 & 20 & 44 & 69 & 76 & 51
 \end{bmatrix}}
  \end{align*} Relaxing the prime $q$, for all $q\geq 17$, we have \[D_{2,q}^{\otimes 6}\equiv 15D_{0,q}\oplus 36D_{2,q}\oplus 40D_{4,q}\oplus 29D_{6,q}\oplus 15D_{8,q} \oplus 5 D_{10,q} \oplus  D_{12,q}.\]
\end{eg}

%\section*{Declarations}

%\subsection*{Conflict of Interest \nopunct}The authors declare no conflict of interest.

\bibliographystyle{amsplain}
\bibliography{bib}
\end{document}